\documentclass[12pt,reqno]{amsart}
\usepackage{diagrams}
\usepackage{amssymb}
\usepackage{mathrsfs}
\usepackage{cite}
\usepackage{tikz-cd}
\usepackage{MnSymbol}
\usepackage{a4wide}

\DeclareMathOperator\C{\mathbb C}
\DeclareMathOperator\Z{\mathbb Z}

\newtheorem{theorem}{Theorem}[section]
\newtheorem{lemma}[theorem]{Lemma}

\newtheorem{prop}[theorem]{Proposition}
\theoremstyle{definition}
\newtheorem{definition}[theorem]{Definition}
\newtheorem{example}[theorem]{Example}

\theoremstyle{remark}
\newtheorem{remark}[theorem]{Remark}
\theoremstyle{remarks}
\newtheorem{remarks}[theorem]{Remarks}

\numberwithin{equation}{section}

\newcommand{\dontprint}[1]\relax

\renewcommand{\Re}{\operatorname{Re}}

\newcommand{\coker}{\operatorname{coker}}

\newcommand{\De}{\Delta}
\newcommand{\La}{\Lambda}
\newcommand{\Ga}{\Gamma}
\newcommand{\Aut}{\operatorname{Aut}}
\newcommand{\End}{\operatorname{End}}
\newcommand{\und}{\underline}
\newcommand{\Pic}{\operatorname{Pic}}
\newcommand{\hra}{\hookrightarrow}
\newcommand{\we}{\wedge}

\renewcommand{\P}{{\mathbb P}}
\newcommand{\A}{{\mathbb A}}

\newcommand{\wt}{\widetilde}
\newcommand{\ot}{\otimes}
\newcommand{\Hom}{\operatorname{Hom}}
\newcommand{\Ext}{\operatorname{Ext}}

\newcommand{\Om}{\Omega}

\newcommand{\XX}{{\mathcal X}}
\newcommand{\YY}{{\mathcal Y}}
\newcommand{\ZZ}{{\mathcal Z}}

\newcommand{\VV}{{\mathcal V}}
\newcommand{\DD}{{\mathcal D}}

\newcommand{\EE}{{\mathcal E}}
\renewcommand{\SS}{{\mathcal S}}
\newcommand{\FF}{{\mathcal F}}

\newcommand{\Th}{\Theta}

\newcommand{\Bun}{\operatorname{Bun}}

\newcommand{\LL}{{\mathcal L}}
\newcommand{\MM}{{\mathcal M}}
\newcommand{\OO}{{\mathcal O}}
\newcommand{\PP}{{\mathcal P}}
\newcommand{\UU}{{\mathcal U}}
\newcommand{\WW}{{\mathcal W}}
\newcommand{\si}{\sigma}
\newcommand{\de}{\delta}
\newcommand{\sub}{\subset}

\newcommand{\ov}{\overline}
\newcommand{\im}{\operatorname{im}}
\newcommand{\om}{\omega}
\newcommand{\la}{\lambda}
\renewcommand{\a}{\alpha}
\renewcommand{\b}{\beta}

\newcommand{\GL}{\operatorname{GL}}

\newcommand{\G}{{\mathbb G}}
\renewcommand{\th}{\theta}
\newcommand{\ga}{\gamma}
\newcommand{\lan}{\langle}
\newcommand{\ran}{\rangle}

\newcommand{\rk}{{\operatorname{rk}}}
\newcommand{\codim}{{\operatorname{codim}}}

\newcommand{\SL}{{\operatorname{SL}}}

\newcommand{\fm}{{\frak m}}
\newcommand{\fg}{{\frak g}}
\newcommand{\Q}{{\mathbb Q}}
\newcommand{\vol}{{\operatorname{vol}}}

\title[Schwartz $\kappa$-densities near stable bundles]{Schwartz $\kappa$-densities on the moduli stack of rank $2$ bundles near stable bundles}

\author{David Kazhdan}
\author{Alexander Polishchuk}
\thanks{D.K. is partially supported by the ERC grant 101142781.
A.P. is partially supported by the NSF grant DMS-2349388, by the Simons Travel grant MPS-TSM-00002745,
and within the framework of the HSE University Basic Research Program.}

\begin{document}

\begin{abstract}
Let $C$ be a curve over a non-archimedean local field of characteristic zero.
We formulate algebro-geometric statements that imply boundedness of functions on the moduli space of stable bundles
of rank $2$ and fixed odd degree determinant over $C$, coming from the Schwartz space of $\kappa$-densities on the corresponding stack of bundles
(earlier we proved that these functions are locally constant on the locus of very stable bundles).
We prove the relevant algebro-geometric statements for curves of genus $2$ and for non-hyperelliptic curves of genus $3$.
\end{abstract}

\maketitle

\section{Introduction}

We work over a non-archimedean local field $k$ of characteristic zero.
Let $C$ be an irreducible smooth proper curve over $k$. 

The analytic Langlands program is concerned with the study of Hecke operators on certain spaces of functions 
associated with the stack of $G$-bundles on $C$ (see \cite{EFK1, EFK2, EFK3},
\cite{BK} and \cite{BKP}). It is conjectured in this program that the Hecke operators act on appropriate $L^2$-spaces
and that their eigenfunctions belong to the Schwartz space of half-densities on the stack of $G$-bundles defined in \cite{GK}.

Let $\Bun_{\La_0}$ denote the moduli stack of rank $2$ bundles over $C$ with determinant $\La_0$, and let $\MM_{\La_0}$ be the coarse moduli space
of stable bundles of rank $2$ with determinant $\La_0$.
In \cite{BKP-Schwartz} we considered the Schwartz spaces $\SS(\Bun_{\La_0},|\om|^{\kappa})$ (defined as in \cite{GK}), where $\kappa\in \C$, and defined natural maps
for $\Re(\kappa)\ge 1/2$,
$$\pi_{\kappa}:\SS(\Bun_{\La_0},|\om|^{\kappa})\to C^\infty(\MM^{vs}_{\La_0},|\om|^{\kappa}),$$
given by integration over orbits, where $\MM^{vs}_{\La_0}$ is the open subset of very stable bundles, $C^\infty(\cdot)$ denotes the space of locally constant sections.

In the present work,
assuming $\deg(\La_0)$ is odd,
we study the behavior of locally constant functions on $\MM^{vs}_{\La_0}$ obtained as above near stable (but not very stable) bundles.
We pose the following conjecture about this behavior.

\medskip

\noindent
{\bf Conjecture A}. {\it Assume $\deg(\La_0)$ is odd.
For $\Re(\kappa)> 1/2$ or $\kappa=1/2$, and any $\mu\in \SS(\Bun_{\La_0},|\om|^{\kappa})$, the locally constant $\kappa$-density $\pi_{\kappa}(\mu)$ on $\MM^{vs}_{\La_0}$ extends to
a continuous section of $|\om|^{\kappa}$ on $\MM_{\La_0}$. 
}

\medskip

We will check this for genus $2$ and $\kappa=1/2$ (as a consequence of Theorems \ref{main-conj-thm} and \ref{ConjB-g2-thm}).
We also consider a weaker conjecture, in which continuity is replaced by a weaker notion of a bounded ($\kappa$-twisted) distribution (see Def.\ \ref{bounded-def}).

\medskip

\noindent
{\bf Conjecture A'}. {\it  Assume $\deg(\La_0)$ is odd.
For $\Re(\kappa)\ge 1/2$ and any $\mu\in \SS(\Bun_{\La_0},|\om|^{\kappa})$, $\pi_{\kappa}(\mu)$ extends to
a bounded distribution in $\DD(\MM_{\La_0},|\om|^{\kappa})$. In particular, for $\kappa=1/2$, $\pi_{1/2}(\mu)$ is an $L^2$-half-density on $\MM_{\La_0}$.
}

\medskip

We will check Conjecture A' for non-hyperelliptic curves of genus $3$ (as a consequence of Theorems \ref{main-conj-thm} and \ref{ConjB-g3-thm}).

In general, we formulate a set of algebro-geometric conjectures related to Brill-Noether loci associated with stable bundles of rank $2$ and degree $2g-1$, 
most notably, Conjecture B in  Sec.\ \ref{BT-sec} stating
that a certain scheme $F_E$ associated with such a bundle $E$
has rational singularities. More precisely, $F_E$ is the total space of the family of projective spaces $\P H^0(C,\xi\ot E)$, as $\xi$ varies through line bundles of degree $0$.
Our main result, Theorem \ref{main-conj-thm}, is that these algebro-geometric statements imply Conjecture A'.
We make some partial checks for these algebro-geometric statements (including verification for curves of genus $2$ and non-hyperelliptic curves of genus $3$).
Debarre in \cite{D} gave another proof of Conjecture B for genus $2$ as well as a proof of the fact that each scheme $F_E$ is irreducible of dimension $g$.

Our main tool to deduce results about distributions over $p$-adic fields from algebro-geometric results is
the Aizenbud-Avni's theorem on continuity of push-forwards of smooth distributions under flat morphisms whose fibers have rational singularities
(see \cite{AA}). We deduce from this theorem a slightly more general version involving bounded distributions (see Prop.\ \ref{push-forward-prop}).
On the other hand, we use a geometric picture developed by Bertram \cite{Bertram} and Thaddeus \cite{Thaddeus} that tells how
to go from the projective space of extensions to the moduli space $\MM_{\La_0}$ through a series of controled birational transformations.

The difficulty of checking our conjecture on rationality of singularities of $F_E$ is the absence of any obvious resolution of singularities.
The way we prove this for genus $2$ and $3$ is by listing possible types of singularities explicitly, showing that each of them can be specialized 
to a normal toric singularity known to be rational, and using Elkik's theorem that deformations of rational singularities are rational (see \cite{Elkik}). 

\section{Push-forwards of Schwartz $\kappa$-densities}

\subsection{Background on Schwartz spaces and distributions}

For a smooth variety $X$ over $k$, a line bundle $L$ on $X$, and a complex number $\kappa$,
we denote by $|L|^{\kappa}$ the complex line bundle over $X(k)$, obtained from $L$ and from the character
$|\cdot |^{\kappa}:k^*\to \C^*$. The transition functions of $|L|^{\kappa}$ are locally constant, so one can
define the space $C^\infty(X(k),|L|^|\kappa|)$ (resp., $\SS(X(k),|L|^{\kappa})$) of locally constant sections of $|L|^{\kappa}$ (resp., with compact support),
More generally, we can consider similar spaces fo $|L_1|^{\kappa_1}\ot |L_2|^{\kappa_2}$, etc.
For a nonvanishing section $s:\OO_X\to L$ we denote by $|s|^{\kappa}$ the corresponding trivialization of $|L|^{\kappa}$. 

We define the space of distributions $\DD(X(k),|L|^{\kappa})$ as the dual to $\SS(X(k),|\om_X|\ot |L|^{-\kappa})$. 
For an open subset $U\sub X$, there is a natural restriction map $\DD(X,|L|^{\kappa})\to \DD(U,|L|_U|^{\kappa})$, in particular, we can talk about the support
of a distribution. 

The reason for the above notation is that the space of appropriate (e.g., locally constant) sections of $|L|^{\kappa}$ embeds into $\DD(X(k),|L|^{\kappa})$,
as we will recall now.


Recall that for every global nonvanishing section $\eta\in \om_X$, the corresponding density $|\eta|$ defines a positive measure on $X(k)$.
We have natural spaces $L^1(X(k),|\om_X|)$ (resp., $L^1_{loc}(X(k),|\om_X|)$) of (resp., locally) absolutely integrable densities defined as follows.
Let $U|\sub X$ be a dense open subvariety with a nonvanishing section $\eta\in \Ga(U,\om_X)$. Then for any $\nu\in \Ga(X(k),|\om_X|)$ on $X(k)$, we write
$\nu|_U=f\cdot |\eta|$, and require the positive density $|f|\cdot |\eta|$ to be integrable on $U(k)$ (resp., on $U(k)\cap K$ for any compact open $K\sub X(k)$). This condition does not depend on a choice
of $(U,\eta)$. We identify $L^1$ or $L^1_{loc}$ densities which differ on a subset of measure zero, so that we have embeddings
$$L^1(X(k),|\om_X|)\sub L^1_{loc}(X(k),|\om_X|)\sub \DD(X(k),|\om_X|).$$

\begin{definition} More generally, we define a subspace $L^1_{loc}(X(k),|L|^{\kappa})\sub \DD(X(k),|L|^{\kappa})$ as the subspace of distributions $\nu$, for which 
there exists an open subvariety $U\sub X$, trivializations $\eta_U$ of $\om_U$ and
$s_U$ of $L|_U$ and a function $f$ on $U(k)$, such that $|f|\cdot |\eta_U|$ is integrable on $U(k)\cap K$ for any compact open $K\sub X(k)$, and
for any $\varphi\in \SS(X(k),|\om_X|\ot |L|^{-\kappa})$, one has
$$\nu(\varphi)=\int f |s_U|^{\kappa}\varphi$$
 (this condition does not depend on a choice of $\eta_U$ and $s_U$).
\end{definition}

We have embeddings
$$C^\infty(X(k),|L|^{\kappa})\sub C^0(X(k),|L|^{\kappa})\sub L^1_{loc}(X(k),|L|^{\kappa})\sub\DD(X(k),|L|^{\kappa}),$$
where $C^0$ denotes the space of continuous sections.



\begin{definition}\label{bounded-def}
We say that a distribution $\nu\in \DD(X(k),|L|^{\kappa})$ is
{\it bounded} if for every point $x$, and for trivializations $s$ of $L$ and $\eta$ of $\om_X$ near $x$, there exists a compact neighborhood $B$ of $x$ and a positive constant $C$,
such that we have 
$$|\nu(f\cdot |\eta|/|s|^{\kappa})|\le C\cdot \int |f|\cdot |\eta|$$
for every locally constant function $f$ supported on $B$.
\end{definition}

\begin{lemma}
(i) If $\nu\in \DD(X(k),|L|^{\kappa})$ is bounded then for any $g\in C^\infty(X(k))$, $\nu\cdot f$ is also bounded.

\noindent
(ii) Let $X$ be a smooth $k$-variety, $U\sub X$ an open subvariety, $\nu\in \DD(X(k),|L|^{\kappa})$ a bounded element, such that $\nu|_U$ is locally constant.
Then for any trivialization $s$ of $L$ on an open $V\sub X$, and any compact $K\sub V(k)$, 
the function $\nu|_U/|s|^{\kappa}$ on $U(k)\cap K$ is bounded.

\noindent
(iii) Assume in addition $X$ is proper, and $\nu\in\DD(X(k),|\om|^{1/2})$ is a bounded element such that $\nu|_U$ is locally constant. Then $\nu|_U$ is in $L^2$.
\end{lemma}

\begin{proof}
(i) We have for $f$ with compact support $B$,
$$|\nu(g\cdot f\cdot |\eta|/|s|^{\kappa})|\le C\cdot \int |g|\cdot |f|\cdot |\eta|\le C\cdot \max_B |g|\cdot \int |f|\cdot |\eta|.$$

\noindent
(ii) It is enough to prove this with $K$ being a sufficiently small open compact neighborhood of any point $x\in V(k)$. In particular, we can assume that there exists a trivialization
$\eta$ of $\om_X$ over $V$.
Since $\nu$ is bounded, we can assume that there exists
a positive constant $C$ such that $\nu(f|\eta|/|s|^{\kappa})|\le C\cdot \int |f|\cdot |\eta|$, for any locally constant function $f$ supported on $K$.
We have $\nu|_U=g\cdot |s|^{\kappa}$, with $g$ locally constant on $U(k)$.
For a given point $y\in U(k)\cap K$, let $B_y\sub U(k)\cap K$ be an open compact neighborhood such that $g|_{B_y}$ is constant. Then
we have 
$$|\nu(\de_{B_y}\cdot |\eta|/|s|^{\kappa})=|g(y)|\cdot \vol_{|\eta|}(B_y)\le C\cdot \vol_{|\eta|}(B_y),$$
hence $|g|\le C$ on $U(k)\cap K$.

\noindent
(iii) Let $X(k)=\sqcup V_i$ be a disjoint covering by compact open subsets, such that each $V_i$ is contained in a Zariski open $U_i$, such that there exists a trivialization $\eta_i$ of
$\om_X$ on $U_i$. Then we can write $\nu=\sum \de_{V_i}\cdot\nu$, and it is enough to check that each $\de_{V_i}\cdot\nu$ is in $L^2$.
By (i), each $\de_{V_i}\cdot \nu$ is bounded, so by (ii), we have
$\de_{V_i}\cdot \nu|_U=g\cdot |\eta_i|^{1/2}$, with $g$ locally constant, bounded and supported on $U(k)\cap V_i$. Hence,
$|g|^2$ is also such, and so $|g|^2\cdot |\eta_i|$ is integrable (since $\mu_{|\eta_i|}(U(k)\cap V_i)<\infty$).
\end{proof}

For an admissible stack $\XX$ over $k$ and a line bundle $L$ on $\XX$, one defines the Schwartz space $\SS(\XX(k),|L|^{\kappa})$ using smooth presentations
surjective on points (see \cite{GK}). In particular, if $f:X\to \XX$ is smooth then we have a well defined push-forward map
$$f_*:\SS(X(k),|\om_f|\ot |f^*L|^{\kappa})\to \SS(\XX(k),|L|^{\kappa}).$$
In the case $\XX=[X/G]$, with $G=\GL_N$, the space $\SS(\XX(k),|L|^{\kappa})$ can be identified with the space of $G(k)$-coinvariants of 
$\SS(X(k),|L|^{\kappa}\ot |{\bigwedge}^{top}\fg|^{-1})$.  


\subsection{Push-forwards for varieties}\label{push-forward-sec}


Let $f:X\to Y$ be a morphism between smooth varieties. Sometimes it is possible to define push-forwards of distributions along $f$.
Let $L$ be a line bundle on $Y$.

\begin{definition}\label{int-def}
Given $\nu\in \DD(X(k),|\om_f|\ot |f^*L|^{\kappa})$, let us say that $\nu$ is {\it $f$-integrable} if for any $\varphi\in \SS(Y(k),|\om_Y|\ot |L|^{-\kappa})$,
the distribution $\nu\cdot f^*\varphi\in \DD(X(k),|\om_X|)$ is in $L^1(X(k),|\om_X|)$.
\end{definition}

If $\nu$ is $f$-integrable then we have a well defined push-forward $f_*\nu$ given by
$$f_*\nu(\varphi)=\int \nu\cdot f^*\varphi.$$

For example, if $\nu\in L^1_{loc}(X(k),|\om_f|\ot |f^*L|^{\kappa})$ has compact support then it is $f$-integrable.

\begin{lemma}\label{composition-integrable-lem}
(i) Let $g:T\to X$ be a smooth morphism. Assume $\nu\in C^\infty(T(k),|\om_{fg}|\ot |(fg)^*L|^{\kappa})$ has relatively compact support over $X$, so
we have $g_*\nu\in C^\infty(X(k),|\om_f|\ot |f^*L|^{\kappa})$. If $\nu$ is $fg$-integrable then 
$g_*\nu$ is $f$-integrable. The converse is true if there exist trivializations $s$ of $L$ on $Y$, $\eta_Y$ of $\om_Y$, $\eta_f$ of $\om_f$ and $\eta_g$ of $\om_g$ such that
\begin{equation}\label{nu-f-g-triv-eq}
\nu=\psi\cdot |\eta_f\cdot f^*\eta_g|\cdot |(fg)^*s|^{\kappa},
\end{equation}
with $\psi\ge 0$.
In either case 
$$(fg)_*\psi=f_*(g_*\psi).$$


\noindent
(ii) Let $U\sub Y$ be an open subvariety, and let $f_U:f^{-1}(U)\to U$ denote the induced morphism. 
If $\nu\in \DD(X(k),|\om_f|\ot |f^*L|^{\kappa})$ is $f$-integrable then $\nu|_{f^{-1}(U)}$ is $f_U$-integrable
and $f_*\nu|_U=(f_U)_*(\nu|_{f^{-1}(U)})$.
\end{lemma}

\begin{proof}
(i) Given $\varphi\in \SS(Y(k),|\om_Y|\ot |L|^{-\kappa})$, we look at the density $\eta:=\nu\cdot (fg)^*\varphi\in C^\infty(T(k),|\om_T|)$ with relatively compact support over $X$.
We have a well defined density $g_*\eta\in C^\infty(X(k),|\om_X|)$. If $\eta\in L^1$ then by Fubini's theorem, $g_*\eta\in L^1$ and
$\int \eta=\int g_*\eta$, which proves the first assertion.

Now assume $\nu$ has form \eqref{nu-f-g-triv-eq}, and take $\varphi\cdot |\eta_Y|\cdot |s|^{-\kappa}\in \SS(Y(k),|\om_Y|\ot |L|^{-\kappa})$, where $\varphi$ is locally constant
with compact support. We want to prove that the density on $T$,
$$\eta:=\nu\cdot (fg)^*(\varphi\cdot |\eta_Y|\cdot |s|^{-\kappa})=\psi\cdot (fg)^*\varphi\cdot |\eta_f\cdot f^*\eta_g\cdot (fg)^*\eta_Y|,$$
is in $L^1$. Without loss of generality we can assume that $\varphi\ge 0$. Then $\eta$ is non-negative, so by Fubini's theorem, if  
$g_*\eta=g_*\nu\cdot f^*(\varphi\cdot |\eta_Y|\cdot |s|^{-\kappa})$ is in $L^1$ then so is $\eta$.

\noindent
(ii) This follows easily from definitions.
\end{proof}


\begin{prop}\label{push-forward-prop} 
Let $f:S\to T$ be a flat morphism of smooth varieties such that the fibers of $f$ have at most rational singularities.
Then for a line bundle $L$ on $T$ and a bounded element $\nu\in \DD(S(k),|\om_f|\ot |f^*L|^{\kappa})$ with compact support, the push-forward 
$f_*\nu\in \DD(T(k),|L|^{\kappa})$ is bounded
(with compact support).
\end{prop}

\begin{proof}
Given a point $t\in T$, we can choose trivializations $\eta_T$ of $\om_T$ and $s$ of $L$ near $t$, as well as a locally constant function $g$ supported on a small
compact neighborhood of $t$. Without loss of generality we can assume that $\nu$
is supported on a small compact neighborhood $B$ of a point $s\in S$ such that $f(s)=t$. Let $\eta_f$ be a trivialization of $\om_f$ on a Zariski neighborhood of $s$ containing $B$,
and let $\eta_S:=\eta_f\ot f^*\eta_T$ be the induced trivialization of $\om_S$.
Then there exists a constant $C>0$ such that
$$|f_*\nu(g\cdot |\eta_T|/|s|^{\kappa})|=|\nu(f^*g\cdot |f^*\eta_T|/|f^*s|^{\kappa})|\le C\cdot \int_B |f^*g| |\eta_S|=C\cdot f_*(\de_B\cdot |\eta_S|)(g).$$
By Aizenbud-Avni's theorem (see \cite{AA}), $f_*(\de_B\cdot |\eta_S|)$ has form $g\cdot |\eta_T|$, for a continuous function $g$. Hence, there exists a constant
$C'$ such that $f_*(\de_B\cdot |\eta_S|)(g)\le C'\cdot \int |g| |\eta_T|$, which proves boundedness of $f_*\nu$.
\end{proof}



\begin{prop}\label{bir-mod-prop} 
Let 
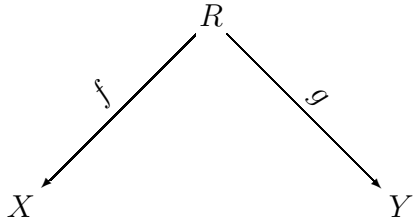
\begin{diagram}
&& R\\
&\ldTo{f}&&\rdTo{g}\\
X&&&&Y
\end{diagram}
be a pair of proper birational morphisms between smooth projective varieties over $k$, such that there exists a Zariski open subset $U\sub R$ mapping isomorphically to 
open subsets in $X$ and $Y$
of codimension $\ge 2$. Let $L$ be a line bundle on $X$, and let $L'$ be the corresponding line bundle $Y$, via the isomorphism
$\Pic(X)\simeq \Pic(U)\simeq \Pic(Y)$.
Assume that $f^*K_X=g^*K_Y(-\sum_i m_iD_i)$ and $f^*L\simeq g^*L'(\sum n_iD_i)$ for effective irreducible divisors $D_1,\ldots,D_r$ on $R$.
Let also $\kappa$ be a complex number with $\Re(\kappa)\cdot n_i\le m_i$ for $i=1,\ldots,r$.
Then, given any $\mu\in C^0(X(k),|\om_X|\ot |L|^{\kappa})$ with compact support, $\mu|_U$, viewed as an element of $\DD(Y(k),|\om_Y|\ot |L'|^{\kappa})$,
is bounded.
\end{prop}

\begin{proof}
For simplicity we assume that $r=1$, i.e., there is only one divisor $D_1$ (the general case is similar).

We can assume that $\mu=\a\cdot |\eta|\cdot |s|^{\kappa}$, 
where $\a$ is a compactly supported continuous function on $X(k)$ and $\eta$ and $s$ are trivializations of $\om_X$ and of $L$, defined on a support of $\a$.
Then $f^*\a$ is a compactly supported continuous function on $X(k)$, and $f^*\eta$ (resp., $f^*s$) is a section of $f^*\om_X$ (resp., $f^*L$)
defined on the support of $f^*\a$. Thus, we can view $f^*\eta$ (resp., $f^*s$)
as a section of $g^*\om_Y$ (resp., $g^*L'$) defined on the support of $f^*\a$. Hence, 
for every point $y\in Y$ and trivializations $\eta'$ of $\om_Y$ and $s'$ of $L'$ in a neighborhood of $y$, we can write
$f^*\eta=h_1\cdot g^*\eta'$, $f^*s=h_2^{-1}\cdot g^*s'$ where $h_1$ and $h_2$ are regular functions  
functions on a neighborhood of the support of $f^*\a$, invertible away from $D$ and vanishing to the order of $m$ and $n$ at $D$, respectively. 
Thus, for any locally constant function $\varphi$ on a compact
neighborhood $B$ of $y\in Y(k)$, the value of $\mu|_U$ on $\varphi\cdot |s'|^{-\kappa}$ is
$$\int_{U(k)\cap B}\varphi\cdot \a\cdot |s/s'|^{\kappa}\cdot |\eta|=\int_{g^{-1}(B)}g^*\varphi\cdot f^*\a\cdot |h_2|^{-\kappa}\cdot |f^*\eta|=
\int_{g^{-1}(B)}g^*\varphi\cdot f^*\a\cdot |h_2|^{-\kappa}\cdot |h_1|\cdot |g^*\eta'|.$$
We claim that the function $f^*\a\cdot |h_2|^{-\kappa}\cdot |h_1|$ is bounded on $g^{-1}(B)$.
Indeed, locally we can write $h_1=u_1\cdot t^m$, $h_2=u_2\cdot t^n$, so the assertion follows from the assumption that
$\Re(m-n\kappa)\ge 0$. Thus, we see that this integral converges and 
there exists a constant $C$ such that
$$|\mu|_U(\varphi)|\le C\cdot \int_{g^{-1}(B)}g^*\varphi\cdot |g^*\eta'|=C\int_B \varphi \cdot \eta',$$
which proves the assertion.
\end{proof}

\subsection{Push-forwards for admissible stacks and an extension result}


Let $X$ be a smooth variety over $k$ with an action of $G=\GL_N$, $L$ a $G$-equivariant line bundle on $X$, $U_0\sub X$ an open $G$-invariant subset,
such that $G$ acts freely on $U_0$ and the quotient $U_0/G$ exists. We pick a trivialization $\vol_{\fg}\in |{\bigwedge}^{top}\fg|^{-1}$.
Let us denote by $\pi_{U_0}:U_0\to U_0/G$ the natural projection.


For each $\varphi\in \SS(X(k),|L|^{\kappa})$ we can consider the restriction $\varphi|_{U_0}$ and ask whether the push forward
$\pi_{U_0,*}(\vol_{\fg}\cdot \varphi|_{U_0})$ is well defined, and if so, whether it is smooth, continuous, or bounded.

\begin{definition}
We say that a pair $(U_0,X)$ is $|L|^{\kappa}$-{\it admissible} (resp., {\it smooth}, resp., {\it continuous}, resp., {\it bounded}) if for every
$\varphi\in \SS(X(k),|L|^{\kappa})$, the distribution $\vol_{\fg}\cdot\varphi|_U$ is $\pi_{U_0}$-integrable in the sense of Definition \ref{int-def} (resp., the push-forward
distribution $\pi_{U_0,*}(\vol_{\fg}\cdot \varphi|_{U_0})$ is smooth, resp., continuous, resp., bounded).
\end{definition}

If $(U_0,X)$ is $|L|^{\kappa}$-{\it admissible} then $\pi_{U_0,*}(\vol_{\fg}\cdot \varphi|_{U_0})$ depends only the image of $\varphi$ in the space of
$G(k)$-coinvariants $\SS(X(k),|{\bigwedge}^{top}\fg|^{-1}\ot |L|^{\kappa})_{G(k)}=\SS([X/G],|L|^{\kappa})$, so we have a map
$$\pi_{X,U_0,\kappa}:\SS([X/G],|L|^{\kappa})\to \DD(U_0/G,|L|^{\kappa}).$$

\begin{lemma}\label{change-of-cover-lem}
Let $f:Y\to [X/G]$ be a smooth morphism, surjective on $k$-points, and let $f_{U_0}:f^{-1}(U_0/G)\to U_0/G$ denote the induced smooth morphism.
Assume that for any $\varphi\in \SS(Y(k),|\om_f|\ot |L|^{\kappa})$, the restriction $\varphi|_{f^{-1}(U_0/G)}$ is $f_{U_0}$-integrable. 
Then the pair $(U_0,X)$ is $|L|^{\kappa}$-admissible and for any $\varphi$ as above, we have
$$\pi_{X,U_0,\kappa}(f_*(\vol_{\fg}\cdot \varphi))=f_{U_0,*}(\vol_{\fg}\cdot\varphi|_{f^{-1}(U_0/G)}).$$
\end{lemma}

\begin{proof}
Replacing $X$, $Y$ and $U_0/G$ by opens of sufficiently fine open coverings, we can assume existence of trivializations $\eta_X$, $\eta_Y$ and $\eta_{U_0/G}$ of $\om_X$,
$\om_Y$ and $\om_{U_0/G}$. We can also assume $L$ to be trivial as a $G$-equivariant line bundle on $U_0$ and as a line bundle on $X$.
Let $\wt{f}:\wt{Y}\to X$ be the morphism from a $G$-torsor $\pi:\wt{Y}\to Y$, corresponding to $f$.
We have the following commutative diagram
\begin{diagram}
&&\wt{Y}&\rTo{\wt{f}}&X\\
&\ldTo{\pi}&\uTo{}&&\uTo{}\\
Y&&\wt{f}^{-1}(U_0)&\rTo{\wt{f'}}&U_0\\
\uTo{}&\ldTo{\pi'}&&\ldTo{\pi_{U_0}}\\
f^{-1}(U_0/G)&\rTo{f_{U_0}}&U_0/G
\end{diagram}
Since the map $\wt{f}_*:\SS(\wt{Y}(k),|\om_f|\ot |L|^{\kappa})\to \SS(X,|L|^{\kappa})$ is surjective, it is enough
to prove that for every $\wt{\varphi}\in \SS(\wt{Y}(k),|\om_{\wt{f}}|\ot |\wt{f}^*L|^{\kappa})$, the restriction of $\wt{f}_*(\wt{\varphi})$ to $U_0$ is
$\pi_{U_0}$-integrable and 
\begin{equation}\label{pi-U0-f-eq}
\pi_{U_0,*}(\wt{f}_*(\vol_{\fg}\cdot \wt{\varphi})|_{U_0})=f_{U_0,*}(\vol_{\fg}\cdot\varphi|_{f^{-1}(U_0/G)}),
\end{equation}
where $\varphi=\pi_*\wt{\varphi}$. In what follows we omit the factors $\vol_{\fg}$ for brevity.

Note that $\wt{\varphi}|_{\wt{f}^{-1}(U_0)}$ has relatively compact support both over
$f^{-1}(U_0/G)$ and over $U_0$, and we have 
$$\pi'_*(\wt{\varphi}|_{\wt{f}^{-1}(U_0)})=\pi_*(\wt{\varphi})|_{f^{-1}(U_0/G)}, \ \ \wt{f}_*(\wt{\varphi}|_{\wt{f}^{-1}(U_0)})=\wt{f}_*(\wt{\varphi})|_{U_0}.$$

By assumption, $\pi_*(\wt{\varphi})|_{f^{-1}(U_0/G)}$ is $f_{U_0}$-integrable. Furthermore, we can assume that 
$$\wt{\varphi}=\psi\cdot |\pi^*\eta_Y/\wt{f}^*\eta_X|\cdot |\wt{f}^*s|^{\kappa},$$ 
where $s$ is a nonvanishing section of $L$, and $\psi\ge 0$.
By Lemma \ref{composition-integrable-lem}(i)
applied to the composition $f_{U_0}\pi':\wt{f}^{-1}(U_0)\to f^{-1}(U_0/G)\to U_0/G$, we deduce that $\wt{\varphi}|_{\wt{f}^{-1}(U_0)}$ is
$f_{U_0}\pi'$-integrable and we have 
\begin{equation}\label{f-U0-pi'-pi-eq}
(f_{U_0}\pi')_*(\wt{\varphi}|_{\wt{f}^{-1}(U_0)})=f_{U_0,*}(\pi_*(\wt{\varphi})|_{f^{-1}(U_0)}).
\end{equation}
Applying Lemma \ref{composition-integrable-lem}(i) to the composition $\wt{f}^{-1}(U_0)\rTo{\wt{f}'} U_0\rTo{\pi_{U_0}} U_0/G$, we deduce that
$\wt{f'}_*(\wt{\varphi}|_{\wt{f}^{-1}(U_0)})$ is $\pi_{U_0}$-integrable and
\begin{equation}\label{two-comp-eq}
(f_{U_0}\pi')_*(\wt{\varphi}|_{\wt{f}^{-1}(U_0)})=(\pi_{U_0}\wt{f}')_*(\wt{\varphi}|_{\wt{f}^{-1}(U_0)})=\pi_{U_0,*}\wt{f}'_*(\wt{\varphi}|_{\wt{f}^{-1}(U_0)}).
\end{equation}

Finally, by Lemma \ref{composition-integrable-lem}(ii), 
\begin{equation}\label{wt-f-res-eq}
\wt{f}'_*(\wt{\varphi}|_{\wt{f}^{-1}(U_0)})=\wt{f}_*(\wt{\varphi})|_{U_0}.
\end{equation}
Now the equality \eqref{pi-U0-f-eq} follows by combining \eqref{f-U0-pi'-pi-eq}, \eqref{two-comp-eq} and \eqref{wt-f-res-eq}.
\end{proof}

Assume now we have another $G$-invariant open $U\sub X$ such that $U_0\sub U$.
We have the following extension result, which allows to change the pair $(U_0,U)$ to $(U_0,X)$ in one of the above definitions.
As a $G$-equivariant line bundle on $X$ we always take 
$$L:=\om_X\ot {\bigwedge}^{top}(\fg).$$

For a point $x\in X(k)$ we denote by $G_x\sub G$ the stabilizer of $x$, and by $\chi_{\om,x}:G_x\to \G_m$ the character given by 
the action on $\om_X|_x$.

\begin{lemma}\label{extension-lem}
Suppose the pair $(U_0,U)$ is $|L|^{k}$-{\it admissible} (resp., {\it smooth}, resp., {\it continuous}, resp., {\it bounded}) for every $k\in (a,b)$, where
$a<1$, $b>1$. Assume also that for any $x\in (X\setminus U)(k)$, there exists a cocharacter $\la^\vee:\G_m\to G_x$ such that
$$k\cdot\lan \chi_{\om,x},\la^\vee\ran\neq -\lan \De^{alg}_{G_x},\la^\vee\ran$$
for any $k\in (a,b)$. Then the pair $(U_0,X)$ is $|L|^{\kappa}$-{\it admissible} (resp., {\it smooth}, resp., {\it continuous}, resp., {\it bounded}) for every $\kappa$ such that
$\Re(\kappa)\in (a,b)$.
\end{lemma}

\begin{proof} The proof is exactly the same as for \cite[Lem.\ 5.2]{BKP-Schwartz}.
\end{proof}

%

\section{From algebraic geometry to results on Schwartz $\kappa$-densities}

\subsection{Bertram-Thaddeus construction}\label{BT-sec}

We denote by $\Bun_d$ the stack of rank $2$ bundles of odd degree $d$ on $C$, and by $\MM_d$ the moduli space of stable rank $2$ bundles of degree $d$.
It will be convenient for us to consider bundles of degree $d=2g-1$.

Let us consider the space $\YY$ of nontrivial extensions
$$0\to \OO\to E\to \La\to 0$$
where $\La$ is some degree $2g-1$ line bundle (which can vary).

There is a natural morphism $\Phi:\YY\to \Bun_{2g-1}$, which induces a rational map from $\YY$ to $\MM_{2g-1}$.
This map admits a resolution of the form
\begin{equation}\label{Th-res-eq}
\YY\lTo{} \YY_1-\!\!\rTo{}\YY_2-\!\!\rTo{}\ldots-\!\!\rTo{}\YY_{g-1}\simeq \MM'_{2g-1}\rTo{} \MM_{2g-1},
\end{equation}
where
\begin{itemize}
\item $\si:\YY_1\to \YY$ is the blow up of the locus of extensions which split over $\La(-p)\sub \La$ for some $p\in C$
(this locus is isomorphic to $C\times J^{2g-1}$ embedded smoothly into $\YY$);
\item
$\YY_i\dashedrightarrow \YY_{i+1}$ are flips, each given as an explicit blow up followed by a blow down;
\item $\MM'_{2g-1}$ is the moduli space of pairs $(E,s:\OO\to E)$, where $E\in \MM_{2g-1}$ and $s\neq 0$, and $\MM'_{2g-1}\to \MM_{2g-1}$
is the natural projection 
\end{itemize}
(this picture is considered in \cite{Bertram} \cite{Thaddeus} for a fixed determinant $\La$).

Now let us return to the moduli with fixed determinant. Let $\La_0$ be a fixed line bundle of degree $2g-1$ on $C$.
Let us consider the space $Y$ of nontrivial extensions 
$$0\to \xi^{-1}\to E\to \La_0\xi\to 0$$
where $\xi\in J$. Note that $Y$ is a projective bundle over $J$ and
we denote by $p_J:Y\to J$ the natural projection.

\begin{lemma} The natural morphism $\phi:Y\to \Bun_{\La_0}$ is smooth.
\end{lemma}

\begin{proof}
As in \cite[Lem.\ 6.2.2]{DG}, this follows from the identification of the cokernel of the tangent map to $\phi$  
with $\Ext^1(\xi^{-1},L_0\xi)\simeq H^1(C,L_0\xi^2)$, which is zero since $\deg(L_0\xi^2)=2g-1>2g-2$.
\end{proof}

We denote by $Y^{st}\sub Y$ the open locus where $E$ is stable, so that we have a commutative diagram
\begin{equation}\label{Y-Yst-diagram}
\begin{diagram}
Y^{st}&\rTo{\phi^{st}}& \MM_{\La_0}\\
\dTo{}&&\dTo{}\\
Y&\rTo{\phi}& \Bun_{\La_0}
\end{diagram}
\end{equation}
with smooth horizontal arrows.




We can use the Bertram-Thaddeus picture to resolve the rational map from $Y$ to $\MM_{\La_0}$.
Let us consider the map $J\to J^{2g-1}: \xi\mapsto \La_0\xi^2$. On ther other hand, we have the determinant map $\Bun_{2g-1}\to J^{2g-1}$.

The morphism $\iota:Y\to \YY$ sending $(\xi,E)$ to $E\ot \xi$ fits into the diagram with cartesian squares
\begin{diagram}
Y&\rTo{\iota}&\YY\\
\dTo{(p_Y,\phi)}&&\dTo{\Phi}\\
J\times \Bun_{\La_0}&\rTo{\a}&\Bun_{2g-1}\\
\dTo{}&&\dTo{}\\
J&\rTo{\b}&J^{2g-1}\\
\end{diagram}
where $\a(\xi,E)=\xi\ot E$, $\b(\xi)=\La_0\xi^2$, and horizontal maps are
 \'etale surjective.
Hence, the \'etale base change of \eqref{Th-res-eq} with respect to $\b$ gives a resolution of
the rational map $Y\dashedrightarrow J\times \MM_{\La_0}$ of the form
\begin{equation}\label{Th-res-eq-bis}
Y\lTo{} Y_1-\!\!\rTo{} Y_2-\!\!\rTo{}\ldots-\!\!\rTo{} Y_{g-1}\simeq Y'\rTo{} J\times \MM_{\La_0},
\end{equation}
where $Y_1\to Y$ is the blow up of the locus of extensions split over $\La_0\xi(-p)$ for some $p\in C$; $Y_i\dashedrightarrow Y_{i+1}$ are flips, each described explicitly as a blow up followed by a blow down; and 
$Y'=\MM'_{2g-1}\times_{J^{2g-1}} J$, which then projects birationally to $J\times\MM_{\La_0}$.

Let $E$ be a fixed stable vector bundle of rank 2 and degree $2g-1$.
Consider the map
$$i_E:J\to \MM_{2g-1}: \xi\mapsto \xi\ot E.$$

\medskip

\noindent
{\bf Conjecture B}. {\it For any $E\in \MM_{2g-1}$, the 
scheme 
$$F_E:=\{(\xi,s) \ | \ \xi\in J, s\in \P H^0(E\xi)\}$$
is irreducible of dimension $g$ and has at most rational singularities.
}

\medskip

Note that tensoring with $\xi_0\in J$ induces an isomorphism of $F_E$ with $F_{\xi_0 E}$, so in Conjecture B it is enough to study singularities of $F_E$ over $\OO\in J$ (for all $E$).

The fiber of the composed map
\begin{equation}\label{Y'-M-map}
Y'=\MM'_{2g-1}\times_{J^{2g-1}} J\to J\times \MM_{\La_0}\to \MM_{\La_0}
\end{equation}
over a stable bundle $E\in \MM_{\La_0}$ is isomorphic to $F_E$. 
Note that Conjecture B implies that the map $Y'\to \MM_{\La_0}$ is flat (see Lemma \ref{conjC-thm} below).
Our main idea is to use the maps $Y\to \Bun_{L_0}$ and $Y'\to \MM_{\La_0}$ to study the behaviour of elements of the Schwartz space $\SS(\Bun_{L_0},|\om|^{\kappa})$,
proving Conjectures A and A' in some cases.

\begin{theorem}\label{main-conj-thm}
Let $C$ be a curve of genus $g\ge 2$ with $C(k)\neq\emptyset$. 
Then Conjecture B implies Conjecture A'.
For $g=2$, Conjecture B implies Conjecture A for $\kappa=1/2$. 
\end{theorem}

The proof will be given in Sec.\ \ref{main-thm-sec}.
Then in Sec.\ \ref{fibers-sec} we will discuss Conjecture B. In particular, we will show that it is true in genus $2$ and for a nonhyperelliptic curve of genus $3$.


\begin{lemma}\label{conjC-thm}
Assume each $F_E$ has dimension $g$. Then
the map $Y'\to \MM_{\La_0}$ is flat.
\end{lemma}

\begin{proof}
This follows from the miracle flatness criterion. 
\end{proof}

\subsection{The relation between line bundles}


We want to apply the results of Section \ref{push-forward-sec} to the resolution of the rational map from $Y$ to $\MM_{\La_0}$. Recall that this resolution starts with
the blow up $\si:Y_1\to Y$, followed be a sequence of flips from $Y_1$ to $Y'$.
Note that we have a regular morphism $\phi_1:Y_1\setminus Z_1\to \MM_{\La_0}$,
for some $Z_1\sub Y_1$ of codimension $\ge 2$.
However, the diagram of natural maps
\begin{diagram}
Y_1\setminus Z_1 &\rTo{\phi_1}& \MM_{\La_0}\\
\dTo{\si}&&\dTo{}\\
Y &\rTo{\phi}& \Bun_{\La_0}
\end{diagram}
is not commutative, although it becomes commutative when restricted to the locus where $\si$ is an isomorphism (this is possible due to non-separatedness of the stack $\Bun_{\La_0}$).
This leads to the following discrepancy in pull-backs of natural line bundles to $Y_1$ from $\Bun_{\La_0}$ and from $\MM_{\La_0}$.

\begin{lemma}\label{line-bun-discrepancy-lem}
One has a natural isomorphism of line bundles on $Y_1\setminus Z_1$,
\begin{equation}\label{blow-up-line-bun-isom}
\si^*\phi^*(\om_{\Bun_{\La_0}}^{1/2})\simeq \phi_1^*(\om_{\MM_{\La_0}}^{1/2})(-(2g-3)E),
\end{equation}
where $E$ is the exceptional divisor of the blow up $\si:Y_1\to Y$.
Hence, we get an isomorphism
\begin{equation}\label{blow-up-line-bun-isom-bis}
\si^*|\om_Y\ot \phi^*\om_{\Bun_{\La_0}}^{-1}|\ot \si^*\phi^*|\om_{\Bun_{\La_0}}|^{\kappa}\simeq
|\om_{Y_1}\ot\phi^*\om_{\MM_{\La_0}}^{-1}|\ot \phi_1^*|\om_{\MM_{\La_0}}|^{\kappa}\ot |\OO_{Y_1}(-E)|^{(4g-6)\kappa-g+1}.
\end{equation}
\end{lemma}

\begin{proof}
Since the restriction of both sides to $Y_1\setminus E$ is the same, we know that we have an isomorphism of the form \eqref{blow-up-line-bun-isom} with some $mE$,
so we only have to check that $m=2g-3$. 

Since the canonical bundle of $\om_J$ is trivial we can replace the map $\phi$, $\phi_1$ and $\si$ by the corresponding maps
$\Phi:\YY\to \Bun_{2g-1}$, $\Phi_1:\YY_1\setminus \ZZ_1\to \MM_{2g-1}$ and $\si:\YY_1\to \YY$.

Furthermore, it is enough to consider restrictions of both sides of \eqref{blow-up-line-bun-isom} to a projective space $\P^{3g-5}$ which is an exceptional fiber of $\si:\YY_1\to \YY$, corresponding
to nontrivial extensions $0\to \OO(p)\to E_1\to \La(-p)\to 0$ with fixed $\La$ of degree $2g-1$ and $p\in C$ (the corresponding point in $\YY$ is the unique nontrivial extension 
$0\to \OO\to E\to \La\to 0$ that splits on $\La(-p)\sub \La$). The corresponding family $\EE_1$ on $C\times\P^{3g-5}$ fits into an exact sequence 
$$0\to \OO_C(p)\boxtimes \OO(1)\to \EE_1\to \La(-p)\boxtimes \OO\to 0$$
(see \cite[Lem.\ 3.5]{Bertram}).
Hence, the bundle $\und{\End}_0(\EE_1)$ has subquotients $\La(-2p)\boxtimes \OO(-1)$, $\OO$ and $\La^{-1}(2p)\boxtimes \OO(1)$.
This leads to an isomorphism of line bundles on $\P^{3g-5}$.
\begin{equation}\label{det-6H-eq}
\phi_1^*\om_{\MM_{2g-1}}\simeq \det R\pi_*(\und{\End}_0(\EE_1))\simeq \OO([\chi(\La^{-1}(2p))-\chi(\La(-2p))])=\OO(-2(2g-3))
\end{equation}
Since the restriction of the left-hand side of \eqref{blow-up-line-bun-isom} to $\P^{3g-5}$ is trivial, we get $m=2g-3$.

The relation \eqref{blow-up-line-bun-isom-bis} now follows using the standard isomorphism
$$\om_{Y_1}\simeq \si^*\om_Y((3g-5)E).$$
\end{proof}

\begin{lemma}\label{thaddeus-bir-lem}
The rational map $Y_1\dashedrightarrow Y'$ satisfies assumptions of Prop.\ \ref{bir-mod-prop} with respect to the line bundle $L=\phi_1^*\om_{\MM_{\La_0}}^{-1}$ for $\Re(\kappa)\le 1/2$.
Hence, for such $\kappa$, a continuous section of $|\om_{Y_1}|\ot |L|^{\kappa}$ gives a bounded distribution in $\DD(Y'(k),|\om_Y|\ot |L'|^{\kappa})$, where $L'$ corresponds to $L$. 
\end{lemma}

\begin{proof}
In \eqref{Th-res-eq-bis} this rational map is presented as a composition of flips $Y_{i-1}\dashedrightarrow Y_i$ (where $Y_{g-1}=Y'$),  
so it is enough to check the statement
for each flip $Y_{i-1}\dashedrightarrow Y_i$. Let $L_i$ denote the line bundle on $Y_i$ corresponding to $L$.

It follows from \cite{Thaddeus} (after the \'etale base change by $\b$) that there exist smooth subvarieties $Z^-_i\sub Y_{i-1}$ and $Z^+_i\sub Y_i$
such that the blow ups $B_{Z^-_i}Y_{i-1}$ and $B_{Z^+_i}Y_i$ are isomorphic (so that the exceptional loci get identified)
and $\dim Z_i^+=3g-3-i$, $\dim Z_i^-=2i-1$. Thus, we get
$$K_{Y_{i-1}}\simeq K_{Y_i}(-(3g-3i-2)E),$$
where $E$ is the exceptional divisor on the common blowup.

Next, we have to calculate the integer $p$ such that $L_{i-1}\simeq L_i(pE)$.
For this it is enough to restrict to fibers $M_{i-1}$ and $M_i$ of $Y_{i-1}$ and $Y_i$ over a fixed point in $J$. 
Thaddeus constructs particular identifications
$$\Z\oplus \Z\to \Pic(M_i):(m,n)\mapsto \OO_{M_i}(m,n)$$
(compatible with the natural isomorphisms $\Pic(M_{i-1})\simeq \Pic(M_i)$) such that the pullback of a positive generator $\om^{-1/2}$ of $\Pic(\MM_{\La_0})$
is $\OO_{M_i}(2,2g-3)$ (see \cite[Sec.\ 5]{Thaddeus}). Now by \cite[(5.6)]{Thaddeus},
$$\OO_{M_{i-1}}(2,2g-3)\simeq \OO_{M_i}((2g-3-2(i-1))E).$$
Thus, $p=2(2g-2i-1)$, and the condition on $\kappa$ from Proposition \ref{bir-mod-prop} becomes
$$2\Re(\kappa)\cdot (2g-2i-1)\le 3g-3i-2.$$
Since $g-i\ge 1$, this follows from $2\Re(\kappa)\le 1$.
\end{proof}

\subsection{The proof of Theorem \ref{main-conj-thm}}\label{main-thm-sec}

Recall that for any irreducible smooth variety $S$ over $k$ with $S(k)\neq\emptyset$, the set $S(k)$ is Zariski dense in $S$ (see \cite[Prop.\ 3.5.75]{Poonen}).
Below we assume $C(k)$ to be nonempty and $g\ge 2$.
Recall that we fixed a line bundle $\La_0$ of degree $2g-1$.

\begin{prop}\label{Ybig-prop}
Every decomposable rank $2$ bundle of the form $A\oplus B$, where $\deg(B)\le g-1$, $AB\simeq\La_0$ and $H^1(AB^{-1})\neq 0$, is in the image
of $\phi:Y(k)\to \Bun_{\La_0}(k)$. 
\end{prop}

\begin{proof} 
Note that the condition $H^1(AB^{-1})\neq 0$ means that $AB^{-1}\simeq K(-D)$ for an effective divisor $D\ge 0$. 
Set $b:=\deg(B)$. Then $\deg(A)=2g-1-b$, so $\deg(D)=2b-1$. In particular, $b\ge 1$.
Our goal is to find $\xi\in J(k)$ and sections $s_1\in H^0(A\ot \xi)$, $s_2\in H^0(B\ot \xi)$ without common zeros.
Equivalently, we are looking for an effective divisor $D'=p_1+\ldots+p_b$ defined over $k$ (the zeroes of $s_2$) such that there exists a global section of 
$K(D'-D)$ not vanishing at any $p_i$ (note that only $D'$ has to be defined over $k$, not the individual points $p_i$). Thus, we need that 
$$h^0(K(D'-D-p_i))<h^0(K(D'-D)), \text{ for } i=1,\ldots,b,$$
or equivalently,
$$h^0(D-D')=h^0(D-D'+p_i), \text{ for } i=1,\ldots,b.$$ 

\noindent
{\bf Step 1}. Assume $h^0(D)\le b-1$. 
In this case, choosing $p_1,\ldots,p_b$ generically (and defined over $k$), we will have $h^0(D-D'+p_i)=h^0(D-D')=0$, 
so the above condition will be satisfied.

From now on we assume that $h^0(D)\ge b$. Note that $h^0(K-D)=h^0(D)-2b+g\ge g-b\ge 1$, so $D$ is a special divisor.

\noindent
{\bf Step 2}. Assume that $C$ is hyperelliptic, and let $|H|$ be the hyperelliptic linear system, $\pi:C\to \P^1$ the corresponding double covering.
Since $D$ is special with $h^0(D)\ge b$, we necessarily have $D\sim (b-1)H+p$ for some $p\in C$. 
In the case $b$ is even we set $D'=F_1+\ldots+F_{b/2}$, where $F_i=\pi^{-1}(\pi(q_i))$ for generic $k$-points $q_1,\ldots,q_{b/2}$ (in particular, no $F_i$
contain $p$). In the case $b$ is odd, we set $D'=p+F_1+\ldots+F_ {(b-1)/2}$ with $F_i=\pi^{-1}(\pi(q_i))$, for generic $k$-points
$q_1,\ldots,q_{(b-1)/2}$.

From now on we assume that $C$ is non-hyperelliptic. 

\noindent
{\bf Step 3}. We claim that either $D\sim p$ for a point $p\in C$, or $|D|$ is base-point-free and $h^0(D)=b$. Indeed, if $|D|$ is not base-point-free then there exists
$p\in C$ (defined over an algebraic closure of $k$) such that $r(D-p)=r(D)=h^0(D)-1\ge b-1$. Since $\deg(D-p)=2b-2$, our claim follows from Clifford's inequality
(applied to $D-p$ and to $D$). 

In the case $D\sim p$, the point $p$ is defined over $k$ and we can take $D'=p$, so from now on we assume that $|D|$ is base-point-free and $h^0(D)=b\ge 2$.
We denote by $f=f_{|D|}:C\to \P^{b-1}$ the corresponding regular morphism. 

\noindent
{\bf Step 4}. We claim that either the map $f$ is birational onto its image, or $b=2$ and $|D|$ is a trigonal system. 
Indeed, since $\deg(D)$ is odd, if $f$ is not birational then it has degree $\ge 3$, so in this case there would exist $p_1,p_2,p_3\in C$ such that
$h^0(D-p_1-p_2-p_3)=h^0(D)-1=b-1$. But if $b>2$ then this would contradict
Clifford's inequality (since $C$ is non-hyperelliptic).

If $b=2$ and $|D|$ is a trigonal system then we can pick any $k$-point $p$ and set $D':=f^{-1}(f(p))-p$, so that $h^0(D-D')=h^0(D-p_i)=1$,
where $D'=p_1+p_2$.

\noindent
{\bf Step 5}. It remains to consider the case when $f$ is birational onto its image.
Let $U\sub C^{b-1}$ denote the non-empty open set of $(q_1,\ldots,q_{b-1})$ such that $h^0(D-q_1-\ldots-q_{b-1})=1$. We have a morphism
$$h:U\to \P H^0(D)\simeq \P^{b-1}$$ 
associating with $q_\bullet=(q_1,\ldots,q_{b-1})$ the unique $h(q_\bullet)\in \P H^0(D)$ vanishing on $q_1+\ldots+q_{b-1}$.
By the General Position Theorem (see \cite[Sec.\ III.1]{ACGH}), $h$ is dominant, and 
for $q_\bullet$ in a non-empty open subset $U_0\sub U$, the hyperplane section of $f(C)\sub \P^{b-1}$ corresponding to $h(q_\bullet)$ consists
of $2b-1$ distinct points, any $b-1$ of which are linearly
independent. This implies that for generic $(q_1,\ldots,q_{b-1})$ (which can be taken to be defined over $k$), the unique divisor $D'=p_1+\ldots+p_b\in |D-q_1-\ldots-q_{b-1}|$ has
the property that any $b-1$ of the points $f(p_1),\ldots,f(p_b)$ are linearly independent (and $D'$ is defined over $k$). This means that 
$$h^0(D-D')=h^0(D-D'+p_i)=1$$
for all $i$, as required.
\end{proof}


As in \cite[Sec.\ 6]{BKP-Schwartz}, we consider open substacks $(\Bun^{\le n}_{\La_0})^0\sub \Bun^{\le n}_{\La_0}\sub \Bun^{\le n}_{\La_0}$, for $n\in \Z$.
Recall that $\Bun^{\le n}_{\La_0}$ consists of $E$ such that any line subbundle $A\sub E$ has $\deg(A)\le n$. For  $(\Bun^{\le n}_{\La_0})^0$ there is an additional
constraint that for any line subbundle $A\sub E$ with $\deg(A)=n$ one has $H^1(A^2\La_0^{-1})=0$.


\medskip

\begin{proof}[Proof of Theorem \ref{main-conj-thm}]
{\bf Step 1}. Recall that $Y^{st}\sub Y$ denotes the stable locus, so that we have a commutative diagram \eqref{Y-Yst-diagram}
We claim that for any $\nu\in \SS(Y(k),|\om_{\phi}|\ot |\phi^*\om|^{\kappa})$, the restriction $\nu|_{Y^{st}}$ is $\phi^{st}$-integrable and
the distribution $\phi^{st}_*(\nu|_{Y^{st}})$ on $\MM_{\La_0}$ is bounded (resp., continuous for $g=2$).

The first step is to replace the pair $(Y,Y^{st})$ by the pair $(Y_1,Y_1\setminus Z_1)$, where $\si:Y_1\to Y$ is the blow up, which is the first step
of Bertram-Thaddeus resolution. 
If we start with $\nu\in \SS(Y(k),|\om_{\phi}|\ot |\phi^*\om_{\MM_{\La_0}}|^{\kappa})$ then isomorphism
\eqref{blow-up-line-bun-isom-bis} shows that for $\Re(\kappa)\ge (g-1)/(4g-6)$ (in particular, for $\Re(\kappa)\ge 1/2$)
$\si^*\nu$ can be viewed as a continuous section of 
$|\om_{Y_1}|\ot |\phi_1^*\om_{\MM_{\La_0}}|^{\kappa-1}$. 

Next, we can replace $Y_1$ by $Y'$ using Lemma \ref{thaddeus-bir-lem}: since $\Re(1-\kappa)\le 1/2$, it tells us that $\si^*\nu$ extends to a bounded distribution 
$\nu'\in \DD(Y'(k),|\om_{Y'}\ot |(\phi')^*\om_{\MM_{\La_0}}|^{\kappa-1}$. 
The last step is taking the push-forward of $\nu'$ with respect to the map $Y'\to \MM_{\La_0}$. By Lemma \ref{conjC-thm},
the latter map is flat, and by Conjecture B, its fibers have rational singularities.
By Proposition \ref{push-forward-prop} we get that the push forward of $\nu'$ is bounded. 

In the case of genus $2$, we have $Y'=Y_1$, so the birational modification is trivial. 
Furthermore, in the case $\kappa=1/2$, $\si^*\nu$ is still a Schwartz section on $Y_1$.
Thus, by Aizenbud-Avni's theorem (see \cite{AA}) we get continuity.

\noindent
{\bf Step 2}. Now we will mimic the argument of the proof of \cite[Thm.\ 6.8]{BKP-Schwartz}.
Consider an admissible open substack of finite type $[X/G]\sub \Bun_{\La_0}$ (where $G=\GL_N$).
Let $X^{\le n}\sub X$ denote the open subset corresponding to $\Bun^{\le n}_{\La_0}$, and let $X^s\sub X$
denote the open subset corresponding to stable bundles. We want to prove that for
$\Re(\kappa)\ge 1/2$, the pair $(X^s,X)$ is $\kappa$-bounded (resp., $\kappa$-continuous if $g=2$
and $\kappa=1/2$). 

Recall that we have a smooth mophism $\phi:Y\to \Bun_{\La_0}$, so its image is open. We can assume that $[X/G]$ contains
the image of $Y$, and denote by $\im(\phi)\sub X$ the corresponding open subset. By Step 1 and Lemma \ref{change-of-cover-lem}, the pair 
$(X^s,X^s\cup \im(\phi))$ is $\kappa$-bounded (resp., $\kappa$-continuous if $g=2$ and $\kappa=1/2$). 
Now we will prove by induction on $n\ge g-1$ that the same is true for the pair $(X^s, X^{\le n}\cup\im(\phi))$.
The base holds since $X^{\le g-1}=X^s$.

For the induction step, for $n\ge g$, we consider the inclusions
$$X^{\le n-1}\sub (X^{\le n})^0\sub X^{\le n},$$
where $(X^{\le n})^0$ is the open subset corresponding to $(\Bun_{\La_0}^{\le n})^0$.
By induction assumption, the pair $(X^s, X^{\le n-1}\cup\im(\phi))$ is $\kappa$-bounded for $\Re(\kappa)\ge 1/2$ (resp., $\kappa$-continuous if $g=2$ and $\kappa=1/2$).
We claim that the same is true for the pair $(X^s, (X^{\le n})^0\cup\im(\phi))$. If $n<(3g-2)/2$ this follows
from \cite[Lem.\ 6.3(iv)]{BKP-Schwartz} stating that in this case $(\Bun_{\La_0}^{\le n})^0=\Bun_{\La_0}^{\le n-1}$. 
Assume now $n\ge (3g-2)/2$. Then by \cite[Lem.\ 6.4]{BKP-Schwartz} we can apply Lemma \ref{extension-lem} to conclude that the
pair $(X^s, (X^{\le n})^0\cup \im(\phi))$ is $\kappa$-bounded for $\Re(\kappa)\ge 1/2$ (resp., $\kappa$-continuous if $g=2$ and $\kappa=1/2$).

Finally, we claim that for $n\ge g$, one has
$$(X^{\le n})^0\cup\im(\phi)=X^{\le n}\cup \im(\phi),$$
which would conclude the induction step.
It is enough to prove that any $E\in \Bun^{\le n}_{\La_0}\setminus (\Bun^{\le n}_{\La_0})^0$ is contained in the image of $\phi$.
By definition such $E$ fits into an exact sequence
$$0\to A\to E\to A^{-1}\La_0\to 0$$
with $\deg(A)=n$ such that $H^1(A^2\La_0^{-1})\neq 0$.
Hence, $E$ specializes to $A\oplus A^{-1}L_0$. Note that $n\ge g$, so $\deg(A^{-1}L_0)\le g-1$. Thus, by Proposition \ref{Ybig-prop}, the bundle $A\oplus A^{-1}L_0$
is in the image of $\phi$, hence, so is $E$.
\end{proof}



\section{Towards Conjecture B}\label{fibers-sec}

\subsection{Degeneracy loci and their partial resolutions}\label{deg-loci-sec}

Let $\phi:\VV\to \WW$ be a morphism of vector bundles on a scheme $S$, where $\rk\VV=m+1$, $\rk \WW=m$. Recall that the $0$th Fitting ideal of $\coker(\phi)$
is locally generated by the $m\times m$-minors of $\phi$. We denote by $Z_\phi\sub S$ the corresponding subscheme of $S$.
On the other hand, we can view $\phi$ as a section $s$ of $p^*\WW\ot \OO(1)$ on the projective bundle $p:\P(\VV)\to S$.
Let $X(\phi)\sub \P(\VV)$ denote the zero locus of $s$. We can think of $X(\phi)\to S$ as the family of projectivizations of the kernels of $\phi|_s$, for $s\in S$.

\begin{prop}\label{deg-blow-up-prop}
Assume that $S$ is integral and $\phi$ is generically surjective. 

\noindent
(i) Let $X'(\phi)\sub X(\phi)$ be the closure of $S\setminus Z_\phi\sub X(\phi)$.
Then there is a natural identification of closed subschemes in $\P(\VV)$,
$$X'(\phi)=B_{Z_\phi}S,$$
where on the right we have the blow up of $S$ at $Z_\phi$. 

\noindent
(ii) Assume in addition that $S$ is smooth. 
Every component of $Z_\phi$ has codimension $\le 2$, and every component of $X(\phi)$ has dimension $\ge \dim S$.
If $X(\phi)$ has dimension $\dim S$ then it is a local complete intersection. If in addition 
$X(\phi)$ is irreducible then $X(\phi)=B_{Z_{\phi}}S$.

\noindent
(iii) Assume that $S$ is smooth irreducible.
For every $p\ge 1$, set $S_p(\phi):=\{s\in S \ |\rk(\phi)\le m-p\}\sub S$. 
Then $X(\phi)$ is irreducible of dimension $\dim S$ if and only if $\codim S_p(\phi)\ge p+1$ for every $p\ge 1$.
Under these assumptions, $X(\phi)=B_{Z_{\phi}}S$. 
\end{prop}

\begin{proof}
(i) Let $\pi:X'(\phi)\to S$ denote the projection.
By definition, the restriction of the composition
$$p^*\WW^\vee\rTo{p^*\phi^\vee} p^*\VV^{\vee}\to \OO(1)$$
to $X'(\phi)\sub \P(\VV)$ vanishes. Hence, $\pi^*\phi$ factors through a map of bundles of rank $m$,
$$\wt{\phi}:\pi^*\WW^\vee\to  \Om^1_p(1)|_{X(\phi)}.$$
Hence, locally the pullback to $X'(\phi)$ of the ideal of $Z_\phi$ is given by one equation, $\det(\wt{\phi})$, which is not a zero divisor
since $X'(\phi)$ is irreducible and generically $\phi^\vee$ is an embedding of a subbundle.

It follows that we have regular morphism $X'(\phi)\to B_{Z_\phi}S$ over $S$. On the other hand, by  \cite[Prop.\ 3.1.1]{PR}, we get a regular morphism
$$B_{Z_\phi}S\to \P(\VV)$$
sending a generic $s$ to the line $\ker(\phi|_s)$. Its image is contained in $X(\phi)$ and contains $S\setminus Z_\phi$. Hence, it is contained in $X'(\phi)$, and we get the assertion.

\noindent
(ii) Let $M_m$ denote the space of $m\times (m+1)$ matrices, $\phi_m:\OO^{m+1}\to \OO^m$ the universal matrix over $M_m$, and $X_m:=X(\phi_m)$
the corresponding standard resolution of the degeneracy locus $Z_m\sub M_m$. Then locally we have a morphism $f:S\to M_m$ such that $\phi$ is the pull-back of $\phi_m$.
Hence, locally $Z_\phi$ can be identified with the preimage of $Z_m$ under $f$.
Similarly, $X(\phi)$ can be identified with the preimage of $X_m$ under the induced map $S\times \P^m\to M_m\times \P^m$.
Equivalently, $Z_\phi$ is given as the intersection of the graph of $f$, $\Ga(f)\sub S\times M_m$ with $S\times Z_m$ in $S\times M_m$, while
$X(\phi)$ is given as the intersection of $\Ga(f)\times \P^m$ with $S\times X_m$ in $S\times M_m\times \P^m$.
It is well known that $\dim Z_m=\dim M_m-2$. Since $\Ga(f)$ is smooth of codimension $\dim M_m$ in $S\times M_m$, this implies our assertion about $Z_\phi$.
Similarly, since $X_m$ is smooth of codimension $m$ in $M_m\times \P^m$, we get that every component of $X(\phi)$ has dimension $\ge \dim X$.
Furthermore, if $\dim X(\phi)=\dim X$ then $X(\phi)$ is a local complete intersection, hence, Cohen-Macaulay. 
If in addition $X(\phi)$ is irreducible then, being generically reduced, it is reduced. Thus, by (i), we get the identification with the blow up.

\noindent
(iii) Assume first that $\codim S_p(\phi)\ge p+1$ for every $p\ge 1$.
Then the preimage of $S_p(\phi)$ in $X(\phi)$ has dimension $\le \dim S-1$ for each $p\ge 1$. Since the complement to the preimage of $S_1(\phi)$
is an open irreducible subset of dimension $\dim S$, and since every component of $X(\phi)$ has dimension $\ge \dim S$, it follows that $X(\phi)$ is irreducible of
dimension $\dim S$. 

Conversely, suppose $X(\phi)$ is irreducible of dimension $\dim S$. Assume that for some $p\ge 1$, we have $\dim S_p(\phi)\ge \dim S-p$. Then
the preimage of $S_p(\phi)$ is a closed subset of dimension $\ge \dim S$. Hence, it should be the whole $X(\phi)$. But this contradicts to the assumption that
$\phi$ is generically surjective.
\end{proof}


\begin{lemma}\label{flatness-lem}
Let $S\to T$ be a smooth morphism of relative dimension $n$, where $T$ is a smooth curve, and let $\phi:\VV\to \WW$ be a morphism of vector bundles on a $S$, where $\rk\VV=m+1$, $\rk \WW=m$. For each $t\in T$, let $\phi_t$ be the corresponding morphism over $S_t$. Suppose for each $t$, one has 
$\dim X(\phi_t)=\dim S_t=n$. Then the morphism $X(\phi)\to T$ is flat and its fiber over $t$ is isomorphic to $X(\phi_t)$. 
\end{lemma}

\begin{proof} The formation of $X(\phi)$ is compatible with the base change, so $X(\phi)_t\simeq X(\phi_t)$. Hence, our assumption implies that
$\dim X(\phi)\le n+1=\dim S$. By Proposition \ref{deg-blow-up-prop}(ii), it follows that $X(\phi)$ is Cohen-Macaulay. Hence, the assertion follows by
the miracle flatness criterion.
\end{proof}

\begin{example} It is easy to see that $X(\phi)$ can be reducible.
Suppose $S=\A^2$, $\phi:\OO^2\to \OO^1$ is given by $\phi(e_1)=x$, $\phi(e_2)=0$. Then $X(\phi)\sub \A^2\times \P^1$ is the union of two components:
one is $(x=0)\times \P^1$ and another is $\A^2\times \{(e_2)\}$. The second component is equal to the blow up (which is isomorphic to $S$ in this case).
\end{example}


Recall that $\MM_{2g-1}$ denotes the moduli space of stable bundles of rank $2$ and degree $2g-1$
and $\MM'_{2g-1}$ denotes the moduli space of pairs $(E,s:\OO\to E)$, where $E\in \MM_{2g-1}$ and $s\neq 0$.
It is known that $\MM'_{2g-1}$ is smooth irreducible and
we have a natural birational morphism
$$\pi=\pi_{2g-1}:\MM'_{2g-1}\to \MM_{2g-1}$$
(see \cite[(3.6)]{Thaddeus}).
Let us denote by $Z\sub \MM_{2g-1}$ the Brill-Noether locus of $E$ such that $H^1(E)\neq 0$, or equivalently $H^0(E^\vee\ot K)\neq 0$ (defined using the appropriate Fitting ideal).

For a fixed $E\in\MM_{2g-1}$, we have an immersion 
$$\iota_E:J\to \MM_{2g-1}:\xi\mapsto \xi\ot E$$ 
such that $Z_E=\iota_E^{-1}Z$ as a subscheme of $J$, and $F_E$ is isomorphic to the fibered product $J\times_{\MM_{2g-1}} \MM'_{2g-1}$.

The variety $\MM'_{2g-1}\to \MM_{2g-1}$ (resp., $F_E\to J$) is of the form $X(\phi)$ for a morphism of vector bundles $\phi:\VV\to \WW$ such that the
complex $[\VV\to \WW]$ represents $Rp_*(\EE)\in D^b(\MM_{2g-1})$, 
where $\EE$ is the universal bundle on $C\times \MM_{2g-1}$, $p$ is the projection to $\MM_{2g-1}$
(resp., it's pull-back $\iota_E^*\phi$ to $J$). Since $\rk(\VV)-\rk(\WW)=1$, this is exactly the situation considered in Proposition \ref{deg-blow-up-prop}.

Since that $\MM'_{2g-1}$ is irreducible of the same dimension as $\dim \MM_{2g-1}$,
we see using Proposition \ref{deg-blow-up-prop} that for each $p\ge 1$, the locus $Z_p$ of $E\in \MM_{2g-1}$ such that $h^1(E)\ge p$ has codimension $\ge p+1$,
and $\MM'_{2g-1}$ is isomorphic to the blow up of $\MM_{2g-1}$ at the locus $Z=Z_1$, which has codimension $2$ in $\MM_{2g-1}$.

We will be more interested in applying Proposition \ref{deg-blow-up-prop} to the schemes $F_E\to J$ for a fixed stable bundle $E\in \MM_{2g-1}$.

\begin{lemma}\label{bun-deg-local-lem} 
For $E\in \MM_{2g-1}$, let $h=h^1(E)$. Then there exists a morphism of vector bundles $\phi:\VV\to \WW$ in a neighborhood of $[\OO]$ on $J$, with
$\rk \VV=h+1$, $\rk \WW=h$, such that $F(E)$ is isomorphic to $X(\phi)$ over a neighborhood of $[\OO]\in J$, and
$Z_E=Z(\phi)$ near $[\OO]$.
The first order terms of $\phi$ near $[\OO]$ are given by the pairing 
$$\mu=\mu_E: H^0(E^\vee K)\ot H^0(E)\to H^0(K)=T^*_{\OO}J.$$
\end{lemma}

\begin{proof} Let $Rp_*\EE\simeq [\VV^{univ}\rTo{d} \WW^{univ}]$, so that the cohomology of the $d|_{[E]}$ is $H^0(C,E)$ and $H^1(C,E)$.
Locally near $[E]\in \MM_{2g-1}$, we can replace  $[\VV^{univ}\to\WW^{univ}]$ by a quasi-isomorphic complex, so that $\rk \VV^{univ}=h$ and $\rk \WW^{univ}=h+1$.
Then $F(E)=X(\phi)$, $Z_E=Z(\phi)$, where $\phi$ is the differential in the restriction $\iota_E^*[\VV^{univ}\to \WW^{univ}]$.
Furthermore, it is known that $\phi$ is given by the natural map
$$H^0(E)\to H^1(E)\ot H^1(\OO)^*$$
(see \cite[Prop.\ 2.10 and Cor.\ 2.11]{Laumon} or \cite[Sec.\ 3.1]{P-BN})
from which $\mu$ is obtained by dualization.
\end{proof}
%

\begin{remark}\label{Z-rem}
Using the identification of $Z$ with the locus where $h^0(E^\vee K)>0$, we get a surjective morphism
$$\pi_{2g-3}:\MM'_{2g-3}\to Z,$$
where $\MM'_{2g-3}$ is the moduli space of pairs $(E,s:\OO\to E)$ with $E\in \MM_{2g-3}$.
Since $\MM'_{2g-3}$ is smooth of dimension $4g-5=\dim Z$ (by \cite[(3.6)]{Thaddeus}), in the terminology of Kempf \cite{Kempf}, $\pi_{2g-3}$ is a well-presented
family of projective spaces. By \cite[Prop.\ 3]{Kempf}, this implies that the map $\MM'_{2g-3}\to Z$ induced by $\pi_{2g-3}$ is a rational
resolution. In particular, $Z$ has rational singularities (and is Cohen-Macaulay).
\end{remark}




\subsection{Irreducibility of $F_E$ and Conjecture C}

We will state another conjecture on stable bundles of rank $2$ and will show how it implies irreducibility of $F_E$.

\medskip

\noindent
{\bf Conjecture C}. Let $E\in \MM_{2g-1}$, $\La:=\det E$.

\noindent
(i) The locus $Z_E$ of $\xi\in J$ such that $H^1(\xi\ot E)\neq 0$ has dimension $\le g-2$.

\noindent
(ii) 
For every $d$, $1\le d\le g-1$, and every $r\ge 1$, the locus of pairs $(\xi,D)\in J\times C^{[d]}$ such that $H^1(\xi^2\La(-D))\neq 0$ and 
$h^0(\xi E(-D))\ge r$
has dimension $\le g-r-1$. 

\medskip

After the preprint of this work appeared, Conjecture C was proved by Debarre in \cite{D}.

\begin{remarks}
1. Let $Z'\sub J\times\MM_{\La_0}$ be the preimage of $Z\sub \MM_{2g-1}$ under the natural \'etale map $J\times \MM_{\La_0}\to \MM_{2g-1}$.
Then the map $Z'\to Z$ is \'etale and $Z_E$, for $E\in \MM_{\La_0}$, are precisely the fibers of the projection $Z'\to \MM_{\La_0}$.
Since $Z$ is Cohen-Macaulay (see Remark \ref{Z-rem}), so is $Z'$, and by miracle flatness, Conjecture C(i) is equivalent to the flatness of the morphism $Z'\to \MM_{\La_0}$.

\noindent
2.
It is possible that for $r\ge 2$ and $2\le d\le g-2$, the condition involving $H^1$ in Conjecture C(ii) is not necessary. On the other hand, 
it is easy to see that for $r=1$ it is necessary in the case when there exists an embedding $L\hra E$ where $L$ is a line bundle of degree $g-1$.
Indeed, then we have a $(g-1)$-dimensional family of $(\xi,D)$ such that $h^0(\xi L(-D))\ge 1$ (by taking $\xi=L^{-1}(D+D')$, where $D'$ is effective of degree $g-1-d$).
We also have examples when it is necessary with $r\ge 2$ and $d=1$ or $d=g-1$.
\end{remarks}



We will show that Conjecture C(ii) for $E\in\MM_{2g-1}$, where $g>2$ implies that $F_E=\{(\xi,s) \ | \ \xi\in J, s\in \P H^0(E\xi)\}$ is irreducible and normal of dimension $g$ (see
Prop.\ \ref{FE-irr-prop}). For $g=2$ and $g=3$ we will directly prove that $F_E$ is irreducible of dimension $g$ (see Prop.\ \ref{FE-irr-g2-prop} and Prop.\ \ref{g3-FE-prop}).
We will also show that Conjecture C(i) holds for general curves (see Theorem \ref{conjC-general-thm}), while the case $d=1$ of Conjecture C(ii) holds
for general curves of genus $g\ge 6$ (see Proposition \ref{Conj-Cii-mercat}).

\begin{lemma}\label{deg0-subbun-lem} Let $E$ be stable bundle of rank $2$ and degree $2g-1$.
Conjecture C(i) for $E$ implies that for $\xi$ in a nonempty open subset of $J$, there exists an embedding of a subbundle $\xi^{-1}\to E$.
\end{lemma}

\begin{proof} It is enough to prove that
for $\xi\in J$ in a nonempty Zariski open subset, one has $H^0(\xi\ot E(-p))=0$ for every $p\in C$. But Conjecture C(i) implies that for every $C$, the
locus $Z_p$ of $\xi$ with $H^0(\xi\ot E(-p))\neq 0$ has dimension $\le g-2$. Hence, the union $\cup_{p\in C}Z_p\sub J$ is a proper closed subset of $J$, and 
the assertion follows.
\end{proof}

\begin{lemma}\label{UE-smooth-lem}
For $E\in \MM_{2g-1}$, let $U_E\sub F_E$ denote the open subset where $\Ext^1(\xi^{-1},E/s(\xi^{-1}))=0$. Then $U_E$ is smooth of dimension $g$.
\end{lemma}

\begin{proof} It is easy to see that associating with a nonzero map $s:\xi^{-1}\to E$ the corresponding quotient $E\to E/s(\xi^{-1})$, induces an isomorphism
between $F_E$ and the Hilbert scheme of quotients $E\to Q$ such that $\rk Q=1$ and $\deg(Q)=2g-1$. Now we use the well known criterion for the
smoothness of the Hilbert scheme and the identification of the Zariski tangent space with $\Hom(\xi^{-1},E/s(\xi^{-1})$.
\end{proof}

\begin{lemma}\label{FE-codim1-lem} 
Conjecture C(ii) for $E$ implies that $F_E$ is of pure dimension $g$, smooth in codimension $1$.
\end{lemma}

\begin{proof}
By Lemma \ref{UE-smooth-lem}, we have a smooth open subset $U=U_E$. 
We claim that $F_E\setminus U$ has dimension $\le g-2$. Indeed, consider the stratum $S_{d,r}\sub F_E\setminus U$, where $d\ge 1$, $r\ge 1$, consisting of $(\xi,s)$ such that
$s$ vanishes exactly on a divisor $D\sub C$ of degree $d$ and $h^0(\xi E(-D))=r$. Then we have an exact sequence
$$0\to \xi^{-1}(D)\to E\to M\to 0$$
where $M\simeq \xi\La(-D)$. By stability of $E$ we have
$$d=\deg(\xi^{-1}(D))\le g-1.$$
Also, we have 
$$\Ext^1(\xi^{-1},E/s(\xi^{-1}))\simeq \Ext^1(\xi^{-1},M)=H^1(\xi^2\La(-D))\neq 0.$$
The fibers of the projection $S_{d,r}\to J\times C^{[d]}$ are $(r-1)$-dimensional projective spaces. By Conjecture C(ii), the image of this projection 
has dimension $\le g-r-1$. Hence, $\dim S_{d,r}\le g-2$, as claimed.

Note also that by Lemma \ref{deg0-subbun-lem}, $U$ is nonempty. Indeed, if $\xi^{-1}\to E$ is an embedding of a subbundle then $E/s(\xi^{-1})$ is a line bundle of degree $2g-1$,
so $\Ext^1(\xi^{-1},E/s(\xi^{-1}))=0$. 
This finishes the proof of our assertion.
\end{proof}

\begin{lemma}\label{connect-lem}
Assume $g>2$. Then $F_E$ is connected for every $E\in \MM_{2g-1}$. 
\end{lemma}

\begin{proof}
Consider $D=p_1+\ldots+p_m$, where $m$ is large and $p_i$ are distinct points. Similarly to the proof of \cite[Thm.\ 2.3]{FL}, it is enough
to prove that the bundle with the fiber $H^0(\xi E(D))^*$ over $\xi\in J$ is ample (where $\xi$ is trivialized at some point $p\in C$). Representing $E$ as an extension of a line bundle by another line bundle,
we reduce to the similar question with $E$ replaced by a line bundle, which is established in the proof of \cite[Lem.\ 2.2]{FL}.
\end{proof}

\begin{prop}\label{FE-irr-prop}
Assume $g>2$. If Conjecture C(ii) holds for $E$ then $F_E$ is an irreducible normal local complete intersection of dimension $g$.
\end{prop}

\begin{proof}
By Lemma \ref{FE-codim1-lem}, $F_E$ is of pure dimension $g$, nonsingular in codimension $1$. Hence, it is a local complete intersection (as a fiber of a morphism 
between smooth varieties with the difference of dimension equal to $g$).
Therefore, by Serre's criterion, it is normal.
Since by Lemma \ref{connect-lem}, $F_E$ is connected, we deduce that it is irreducible.
\end{proof}

For $g=2$, we will check irreducibility of $F_E$ in a different way.

\begin{prop}\label{FE-irr-g2-prop} Assume $C$ has genus $2$. Then for any $E\in \MM_3$, $F_E$ is irreducible. 
\end{prop}

\begin{proof}
Since we can replace $E$ by $\xi E$ with $\xi\in J$, we can assume that $E$ is an extension
$$0\to KL^{-1}\to E\to K\to 0$$
with $\deg(L)=1$. The class of this extension $e$ lives in $H^1(L^{-1})\simeq H^0(KL)^*$. Let $S\sub F_E$ denote the closed subset where
$H^0(\xi KL^{-1})\neq 0$. We will check that $\dim S\le 1$ and that $F_E\setminus S$ is irreducible of dimension $2$. Since every irreducible component of
$F_E$ has dimension $\ge 2$, this will imply that $F_E$ is irreducible.

Note that for every $\xi\in F_E\setminus S$, the induced map $\xi^{-1}\to K$ is nonzero, so we have $\xi^{-1}\simeq K(-D)$ for $D\in C^{[2]}$. Furthermore, by assumption,
$H^0(\xi KL^{-1})=H^0(L^{-1}(D))=0$, or equivalently, $H^1(L^{-1}(D))=0$. We claim that we get an isomorphism of $F_E\setminus S$ with an open subset in $C^{[2]}$ consisting of
$D$ such that $H^0(L^{-1}(D))=0$. Indeed, for such $D$ we have $H^1(\xi KL^{-1})=H^1(L^{-1}D)=0$, so the map $\xi^{-1}\to K$ uniquely lifts to $\xi^{-1}\to E$.

It remains to check that $\dim S\le 1$. We can write $\xi^{-1}\simeq KL^{-1}(-p)$ for some $p\in C$. Then $\xi E$ fits into an exact sequence
$$0\to \OO(p)\to \xi E\to L(p)\to 0.$$
If $L(p)\not\simeq K$ then $H^0(L(p))$ is $1$-dimensional. We claim that there is only finitely many such $p$ for which
$h^0(\xi E)>1$, or equivalently, the connecting homomorphism $H^0(L(p))\to H^1(\OO(p))$ is zero.
Indeed, this connecting homomorphism is obtained by dualization from the pairing
$$H^0(L(p))\ot H^0(K(-p))\rTo{b_{L,p}} H^0(KL)\rTo{e} k.$$
Thus, we need to check that there is only finitely many $p$ such that the image of $b_{L,p}$ is contained in $\ker(e)$.
But the image of $b_{L,p}$ is precisely the $1$-dimensional subspace $H^0(KL(-\tau(p)))$, where $\tau:C\to C$ is the hyperelliptic involution
(so $K(-p)\simeq \OO(\tau(p))$). In other words, we are looking at $p$ such that $\tau(p)$ is in the preimage of $e$ under the map to $\P^1$ given by the linear system $|KL|$,
so there is only finitely many.

On the other hand, if $L(p)\simeq K$ (so $\xi=\OO$) then we claim that the connecting homomorphism $H^0(K)\to H^1(\OO(p))$ is nonzero,
which implies that $h^0(\xi E)\le 2$, so this case contributes at most $\P^1$ as a component of $S$. Indeed, the connecting homomorphism is obtained by dualization
from the composition
$$H^0(K)\ot H^0(K(-p))\to H^0(K^2(-p))\rTo{e} k.$$
But for any nonzero $s\in H^0(K(-p))$, the subspace $H^0(K)\cdot s\sub H^0(K^2(-p))$ is $2$-dimensional, so it cannot be contained in $\ker(e)$.
\end{proof}





\subsection{On some cases of Conjecture C}

\begin{theorem}\label{conjC-general-thm} 
Conjecture C(i) is true for any $E\in \MM_{2g-1}$ on a curve $C$ of genus $\ge 2$ such that $\End(J_C)=\Z$ (so it is true for general curve).
\end{theorem}

\begin{proof}
 Let $E$ be a stable vector bundle in $\MM_{\La_0}$. We define the bundle $E'$ as a generic Hecke transformation of $E$, so that we have exact sequence
\begin{equation}\label{Hecke-E-E'-ex-seq}
0\to E'\to E\to \OO_p\to 0
\end{equation}
We immediately see that $E'$ is semistable.
Let us consider the generalized theta-divisor $\Th_{E'}=\{\xi\in E' \ |\ H^1(\xi\ot E')\neq 0\}$ (it is always of codimension $1$ in $J$ by the result of Raynaud \cite{Raynaud}). Then we necessarily have
$Z_E\sub \Th_{E'}$. We want to check that for an appropriate choice of the Hecke transformation we have $Z_E\neq \Th_{E'}$.

Let us pick a generic $\xi\in J$, so that $h^1(\xi\ot E)=0$, i.e., $H^0(\xi\ot E)$ is $1$-dimensional. Let $s_{\xi}\in H^0(\xi\ot E)$ be a nonzero element. Then for generic point $p\in C$, we have $s_{\xi}|_p\neq 0$.
and we take the line $\ell=\lan s_{\xi}|_p\ran\ot \xi^{-1}|_p\sub E|_p$. Then
the composition 
$H^0(\xi\ot E)\to \xi|_p\ot E|_p/\ell$ will be zero, and we deduce that $h^0(\xi\ot E')=h^0(\xi\ot E)=1$, i.e., $\xi\in \Th_{E'}\setminus Z_E$.

Next, we claim that for a generic choice of $\ell$ and $p\in C$, $E'$ will be stable. Indeed, one can show that there is finitely many subbundles $L\sub E$ of degree $g-1$.
We need to make sure that $\ell\neq L|_p\sub E|_p$. But if $\xi$ is generic then $H^0(\xi\ot L)=0$, so $s_\xi$ projects to a nonzero section $\ov{s}_\xi$ of $\xi\ot (E/L)$. For a generic $p\in C$
we will have $\ov{s}_\xi|_p\neq 0$, so $\ell=\lan s_\xi|_p\ran\neq L|_p$ as required.

Finally, we claim that $\Th_{E'}$ is irreducible.
Indeed, since $\End(J)=\Z$, and $\Th_{E'}$ is algebraically equivalent to $2\th$ on $J$, if it is reducible it has form $\Th_{L}+\Th_{L'}$ for some line bundles $L$ and $L'$ of degree $g-1$.
In other words, it coincides with the generalized theta-divisor $\Th_{L\oplus L'}$. It is known that the map from the moduli space of semistable bundles to the linear system $|2\th|$ is injective,
so this would contradict stability of $E'$.

Since $Z_E$ is properly contained in $\Th_{E'}$ it has dimension $\le g-2$. On the other hand, $Z_E$ is obtained as the intersection of $Z'$ with $\{E\}\times J$ in $\MM_{\La_0}\times J$, so
each component of $Z_E$ has dimension $\ge g-2$.
\end{proof}

\begin{prop}\label{Conj-Cii-mercat} 
(i) Conjecture C(ii) always holds for $d=1$ and $r\le g-2$. 

\noindent (ii)
Conjecture C(ii) for $d=1$ holds for any $E\in\MM_{2g-1}$ on a curve $C$ with Clifford index $\ga_C\ge 2$ and genus $g\ge 6$.
\end{prop}

\begin{proof} (i) We observe that
$H^1(\xi^2\La(-p))\neq 0$ exactly when $\xi^2(-p)\simeq K\La^{-1}$, which is a $1$-dimensional family of $(\xi,p)$. Thus, the dimension of the needed locus is 
$\le 1\le g-r-1$.

\noindent
(ii) We claim that for $r>g-2$ the statement is vacuous. Indeed,
under our assumptions on $C$, one has $h^0(F)\le g-2$ for any stable bundle of rank $2$ and degree $2g-3$ (see \cite[Thm.\ 2.1]{Mercat}).
It follows that for any $p\in C$ and $\xi\in J$ one has $h^0(\xi E(-p))\le g-2$. 
\end{proof} 

\subsection{Singularities of $F_E$: easy cases}

\begin{prop}\label{FE-smooth-prop}
Let $E\in \MM_{2g-1}$.
Assume that the map $\mu: H^0(E^\vee K)\ot H^0(E)\to H^0(K)$ is injective. Then 
$F_E$ is smooth at any point over $\OO\in J$.
\end{prop}

\begin{proof}
Indeed, by Lemma \ref{bun-deg-local-lem}, locally we have an $h\times (h+1)$-matrix of functions $\phi$, where $h=h^0(E^\vee K)$,
such that $F_E=X(\phi)$, and $\mu$ gives the dual to the tangent map of the corresponding morphism $f:J\to M_h$ to the space of $(h+1)\times h$-matrices. It
follows that $f$ is smooth, which implies that the map to the universal resolution $X(\phi)\to X_m$ is smooth, hence, $X(\phi)$ is smooth.
\end{proof}

\begin{prop}
Assume that $E\in \MM_{2g-1}$ is very stable.
Then $F_E$ is smooth of dimension $g$.
\end{prop}
\begin{proof}
By Lemma \ref{UE-smooth-lem}, it is enough to check that for every nonzero $s:\xi^{-1}\to E$ one has $\Ext^1(\xi^{-1},E/s(\xi^{-1}))=0$.
Assume this $\Ext^1$ is nonzero. Let $D\ge 0$ be a divisor of zeros of $s$, so that $s$ defines an embedding of a subbundle $\xi^{-1}(D)\to E$, and
$$E/s(\xi^{-1})\simeq T\oplus \La\xi(-D),$$
where $T$ is a torsion sheaf and $\La=\det(E)$. Then we should have $\Ext^1(\xi^{-1},\La\xi(-D))\neq 0$. By Serre duality, there exists a nonzero morphism
$\La\xi(-D)\to \xi^{-1}K$. But then we get a nonzero nilpotent Higgs field,
$$E\to \La\xi(-D)\to \xi^{-1}K\to EK,$$
which is a contradiction.
\end{proof}





\begin{lemma}\label{toric-sing-lem}
(i) A singularity of the form $x_1x_2=f(x_3,\ldots,x_n)$, where $f$ is nonzero formal power series, is rational.

\noindent
(ii) Let $D\sub X$ be a smooth divisor in a smooth variety $X$ and let $Z\sub D$ be a divisor. Then $B_Z X$ has rational singularities.
\end{lemma}

\begin{proof}
(i) It is well known that we can assume $f$ to be a polynomial.
Using the fact that rational singularities are stable under generization (see \cite[Thm.\ IV]{Elkik}), we reduce to the case when $f$ is a monomial.
But then we have a (normal) toric singularity, hence rational.

\noindent
(ii) Formally locally we can choose coordinates $(x_1,\ldots,x_n)$ on $X$ such that $D$ is given by $(x_1=0)$ and $E\sub D$ is given by $f(x_2,\ldots,x_n)$.
Then the singular point on the blow up will be given by $x_0x_1=f(x_2,\ldots,x_n)$ so it is rational by part (i).
\end{proof}

\begin{lemma}\label{rat-sing-easy-case-lem}
Assume $E\in \MM_{2g-1}$ is such that $h^1(E)=1$ and $\dim_{\OO} Z_E=g-2$ (local dimension at the point $[\OO]\in J$). 
Then 
$Z_E$ is a locally complete intersection in $J$ near $[\OO]$, and 
$$F_E\simeq B_{Z_E} J$$
near $\OO\in J$.
Let $A=\ker(\a:E\to K)$, where $\a$ is a generator of the $1$-dimensional space $\Hom(E,K)$. Assume in addition that $h^0(A)\leq 1$. 
Then $F_E$ has rational singularities over $\OO\in J$.
\end{lemma}

\begin{proof} 
The fact that $Z_E$ is given locally near $[\OO]$ by $2$ equations follows immediately from Lemma \ref{bun-deg-local-lem}. 
By Proposition \ref{deg-blow-up-prop}(iii) (with $m=1$), this implies that $F_E\simeq B_{Z_E} J$. 

To prove the second assertion, we note that the linear parts of the $2$ equations of $Z$ are given by the linear map $H^0(E)\to H^0(K)$ induced by the unique
nonzero map $E\to K$. By assumption, this map is nonzero. Hence, locally there exists a smooth divisor containing 
$Z_E$. Hence, the result follows from Lemma \ref{toric-sing-lem}(ii).
\end{proof}



\subsection{Proof of Conjecture B for genus $2$}

Let $C$ be a curve of genus $2$.
We already proved that for any $E\in \MM_3$, $F_E$ is irreducible of dimension $2$ (see Prop.\ \ref{FE-irr-g2-prop}).
The validity of Conjecture C(i) follows from the well known fact that for $E\in \MM_{\La_0}$ there is only finitely many line bundles $M$ of degree $2$ with $\Hom(E,M)\neq 0$ (see e.g., \cite{Newstead}).
Since $Z_E\sub J$ consists of $\xi$ such that $\Hom(E,\xi^{-1}K)\neq 0$, we see that each $Z_E$ is $0$-dimensional.

\begin{lemma}\label{K-ext-lem}
For any $E\in \MM_{\La_0}$ one has $h^1(E)\le 1$, or equivalently, $h^0(E)\le 2$.
\end{lemma}

\begin{proof}
We can assume that $h^1(E)\neq 0$, so $\Hom(E,K)\neq 0$, so $E$ is an extension
$$0\to L\to E\to K\to 0$$
for some line bundle $L$ of degree $1$. If $H^1(L)=0$ then we get $H^1(E)\simeq H^1(K)$, so it is $1$-dimensional.
Otherwise, we have $L=K(-p)$ for some point $p\in C$. Then we need to check that the coboundary map
$$\de:H^0(K)\to H^1(K(-p))$$
induced by the cup-product with a nonzero class $e\in H^1(\OO(-p))$ is nonzero. But the maps $H^1(\OO(-p))\to H^1(O)$ and
$H^1(K(-p))\to H^1(K)$ are isomorphisms, so the assertion follows from the fact that the pairing
$$H^0(K)\ot H^1(\OO)\to H^1(K)$$
is perfect.
\end{proof}

\begin{theorem}\label{ConjB-g2-thm} 
Conjecture B holds for any curve $C$ of genus $2$.
\end{theorem}

\begin{proof}
This follows from Lemmas \ref{rat-sing-easy-case-lem} and \ref{K-ext-lem}.
\end{proof}

\begin{proof}[Second proof of Conjecture B]
Replacing $E$ by $E^\vee K$, we reduce to considering
a vector bundle $E$ of rank $2$ and degree $1$, which is an extension
$$0\to \OO\to E\to L\to 0$$
where $\deg(L)=1$ and $H^0(L)\neq 0$. We consider the subscheme $Z_E\sub J$ of $\xi$ such that $H^0(\xi E)\neq 0$ (defined by the Fitting ideal).
We need to prove that the blow up $B_{Z_E} J$ has rational singularities.

We choose a generic divisor $D=p_1+p_2$ and a generic point $q$. We can realize $E$ as the kernel
$$0\to E\to \OO(D)\oplus L\rTo{\varphi} \OO(D)|_D\to 0.$$
One can check that for generic choices we will have $H^1(E(q))=0$.
Arguing as before, we need to consider $2\times 2$-minors of the morphism 
$$H^0(\xi(D))\oplus H^0(\xi L(q))\to \xi(p_1)|_{p_1}\oplus \xi(p_2)|_{p_2}\oplus \xi L(q)|_q$$
given locally by the matrix
$$\left(\begin{matrix} 0 & \th_{p_1}(\xi) & \th_{p_2}(\xi) \\ \th_{L}(\xi) & \varphi_1 & \varphi_2 \end{matrix}\right),$$
where either $\varphi_1$ or $\varphi_2$ is invertible. It follows that the subscheme $Z_E$ is contained in a divisor $D\sub J$ given by the equation
$$\varphi_2\th_{p_1}(\xi)-\varphi_1\th_{p_2}(\xi)=0.$$
Since the theta-divisors $\th_{p_i}(\xi)=0$, $i=1,2$, are smooth and transversal at $\xi=0$, we see that $D$ is smooth at $\xi=0$.
Assuming that $\varphi_1$ is invertible, we see that the subscheme $Z_E$ is cut out in $D$ by the single equation $\th_L(\xi)\th_{p_1}(\xi)=0$.

On the other hand, we know that $Z_E$ is $0$-dimensional. So we can assume that $D$ is given by $y=0$ in $\A^2$ with coordinates $x,y$, and $Z_E$ is given by the ideal $(x^d,y)$.
Then the singularity of the blow up looks like $x^d=yt$ which is a rational type A singularity.
\end{proof}



\subsection{Proof of Conjecture C(i) and of irreducibility of $F_E$ for genus $3$}

Let $C$ be a curve of genus $3$.

\begin{lemma}\label{h12-ext-lem} 
Let $E$ be a nontrivial extension
$$0\to L_2\to E\to L_3\to 0$$
for line bundles $L_2$ of degree $2$ and $L_3$ of degree $3$.
Then the dimension of the locus $Z_E$ of $\xi\in J$ such that $h^1(E\ot \xi)>0$ is $\le 1$.
\end{lemma}

\begin{proof} Switching from $E$ to $F=E^\vee K$, we see that $F$ is a nontrivial extension
$$0\to L_1^{-1}L_2\to F\to L_2\to 0$$
for some line bundles $L_1$ of degree $1$ and $L_2$ of degree $2$, and we need to study the locus of $\xi\in J$ such that $h^0(\xi F)\neq 0$.
We have $H^0(\xi L_1^{-1}L_2)=0$ and $h^0(\xi L_2)\le 1$ for $\xi$ away from a locus of dimension $\le 1$, so it suffices to consider only such $\xi$.
We have to show that the coboundary homomorphism
$$H^0(\xi L_2)\to H^1(\xi L_1^{-1}L_2)$$
is nonzero away from a locus of dimension $\le 1$. We have $\xi L_2=\OO(D)$ for a unique effective divisor $D$ of degree $2$.
By Serre duality, we have to check nonvanishing of the pairing
$$H^0(\OO(D))\ot H^0(KL_1(-D))\to H^0(KL_1)\rTo{e} k$$
for generic $D$, where $e\in H^0(KL_1)^*\simeq H^1(L_1^{-1})$ is the class of our extension.

It is enough to show that the subspaces $H^0(KL_1(-D))$ over all $D$ span the entire $3$-dimensional space $H^0(KL_1)$.
To this end let us pick three generic points $p_1,p_2,p_3\in C$ such that $H^0(KL_1(-p_1-p_2-p_3))=0$.
Then we have $H^0(KL_1(-p_1))=H^0(KL_1(-p_1-p_2))+H^0(KL_1(-p_1-p_3))$ (since these two subspaces have trivial intersection and $H^0(KL_1(-p_1))$ is
$2$-dimensional. Furthermore, $H^0(KL_1(-p_1))$ has trivial intersection with $H^0(KL_1(-p_2-p_3))$, hence
\begin{align*}
&H^0(KL_1)=H^0(KL_1(-p_1))+H^0(KL_1(-p_2-p_3))=\\
&H^0(KL_1(-p_1-p_2))+H^0(KL_1(-p_1-p_3))+H^0(KL_1(-p_2-p_3)).
\end{align*}
\end{proof}


\begin{prop}\label{Conj-Ci-g3-prop}
Conjecture C(i) holds for any curve of genus $3$, i.e., for any stable bundle $E$ of rank $2$ and degree $5$ the locus $Z_E$ has dimension $\le 1$ (and hence,
is equidimensional of dimension $1$).
\end{prop}

\begin{proof} If $E$ is an extension of a line bundle of degree $3$ by a line bundle of degree $2$ then this follows from Lemma \ref{h12-ext-lem}.

Now assume that $E$ cannot appear as such an extension, in other words, for any line bundle $L_3$ of degree $3$ one has $\Hom(E,L_3)=0$.
It follows that any nonzero morphism $\xi E\to K$ is automatically surjective. If the kernel of this map $A$ has no global sections then $Z_E$ is smooth of
dimension $1$ at $\xi$. Therefore, it is enough to check that locus of $\xi$ such that $\xi E$ fits into an extension 
$$0\to \OO(p)\to \xi E\to K\to 0$$
for some $p\in C$, has dimension $\le 1$. But this immediately follows from the isomorphism
$$\det(\xi E)\simeq \xi^2 \det(E)\simeq K(p),$$
which gives $\xi^2\simeq \det(E)^{-1}K(p)$ for some $p\in C$.
\end{proof}

\begin{lemma}\label{g3-h1-ext-lem}
Let $E\in \MM_5$ be such that $h^1(E)>1$. Then $E$ fits into an extension
\begin{equation}\label{L-E-Kp-ext}
0\to L\to E\to K(-p)\to 0
\end{equation}
for some point $p\in C$ and some line bundle $L$ of degree $2$.
\end{lemma}

\begin{proof}
By Serre duality, we have $h^0(E^\vee K)>1$. Let $\a,\b$ be linearly independent maps $E\to K$ and consider the corresponding map $f=(\a,\b):E\to K\oplus K$.
If the image of $f$ has rank $1$ then it factors through a line subbundle $K(-D)$, where $D>0$. If the image of $f$ has rank $2$, then $f$ factors through a subsheaf isomorphic to $K\oplus K(-p)$.
In either case $E$ has a nonzero map to $K(-p)$. Since $E$ is stable of slope $5/2$, this map has to be a surjection. Hence, $E$ fits into an extension 
of the form \eqref{L-E-Kp-ext}.
\end{proof}

\begin{lemma}\label{g3-EDp-lem}
Let $E$ be a stable vector bundle on $C$ of rank $2$ and degree $5$ with $h^1(E)>1$. 
(i) Assume that $C$ is non-hyperelliptic. 
Then $h^1(E)=2$ and there exist a point $p\in C$ and an effective divisor
$D$ of degree $2$ on $C$, such that 
$E=E_{D,p}$ is the unique nontrivial extension
$$0\to K(-D)\to E_{D,p}\to K(-p)\to 0$$
that splits over $K(-D-p)\sub K(-p)$.

\noindent
(ii) If $C$ is hyperelliptic then $h^1(E)=2$, and either $E=E_{D,p}$ for some $D$ and $p$ with $h^0(D)=1$, or $E$ is an extension
\begin{equation}\label{hyperell-EL-ext}
0\to L\to E\to L(q)\to 0
\end{equation}
where $L$ is the hyperelliptic system, $q\in C$, and the extension class corresponds to a rank $1$ quadratic form under the isomorphism
$$H^1(\OO(-q))\simeq H^1(\OO)\simeq H^0(K)^*\simeq S^2H^0(L)^*.$$
\end{lemma}

\begin{proof}
By Lemma \ref{g3-h1-ext-lem}, $E$ fits into an extension of the form \eqref{L-E-Kp-ext} for some $p\in C$ and some line bundle $L$ of degree $2$.

Assume first that $h^0(L)\le 1$ (which is automatic if $C$ is non-hyperelliptic).
Since $h^0(E)\ge 3$, while $h^0(L)\le 1$ and $h^0(K(-p))=2$, we see that $h^0(L)=1$ and the coboundary map
$\de:H^0(K(-p))\to H^1(L)$ has to be zero. Let us write $L$ as $L=K(-D)$, where $D$ is an effective divisor of degree $2$. Then the class of the extension is an element
of $e\in H^1(\OO(p-D))$. Note that the natural map $H^1(K(-D))\to H^1(K)$ is an isomorphism. Hence, $\de$ is zero if and only if the map
$H^0(K(-p))\to H^1(K)$ given by the class $e'\in H^1(\OO(p))$ is zero, where $e'$ is obtained as the image of $e$ under the natural map $H^1(\OO(p-D))\to H^1(\OO(p))$.
By Serre duality, this happens if and only $e'$ is zero, which immediately gives our statement.

Next, assume $C$ is hyperelliptic and $L$ is the hyperelliptic system. Then the coboundary map $H^0(K(-p))\to H^1(L)$ should have a nonzero kernel. 
We have $K(-p)\simeq L(\tau(p))$, where $\tau$ is the hyperelliptic involution, and $H^0(K-p)=H^0(K(-p-\tau(p)))\simeq H^0(L)$. 
The extension class of $E$ is a nonzero element $e$ in $H^1(\OO(-\tau(p)))\simeq H^1(\OO)\simeq H^0(K)^*$.
By Serre duality, the coboundary map $H^0(L)\to H^1(L)$ corresponds to the pairing
$$H^0(L)\ot H^0(L)\to S^2H^0(L)\rTo{\sim} H^0(K)\rTo{e} k.$$
Thus, $e$ should correspond to a rank $1$ quadratic form on $H^0(L)$.
\end{proof}

\begin{lemma}\label{g3-Cii-lem} 
Assume that $E\in \MM_5$ is not an extension of a line bundle of degree $3$ by a line bundle of degree $2$.
Then Conjecture C(ii) holds for $E$. Hence, $F_E$ is irreducible of dimension $3$.
\end{lemma}

\begin{proof} By assumption, there are no nonzero morphisms from line bundles of degree $\ge 2$ to $E$, so we only need to consider the case
$d=1$ of Conjecture C(ii). The case $r=1$ follows from Proposition \ref{Conj-Cii-mercat}(i).

It remains to prove that for any $p\in C$, one has $h^0(\xi E(-p))\le 1$. Indeed, set $F=K\xi^{-1}E^\vee(p)\in \MM_5$. Then we need to check that $h^1(F)\le 1$.
But $F$ is not an extension of a line bundle of degree $3$ by a line bundle of degree 2. Hence the assertion follows from
Lemma \ref{g3-h1-ext-lem}.
\end{proof}

\begin{prop}\label{g3-FE-prop} For any $E\in \MM_5$,
$F_E$ is irreducible of dimension $3$.
\end{prop}

\begin{proof} 
By Lemma \ref{g3-Cii-lem}, we can assume that
$E$ is a non-trivial extension 
\begin{equation}\label{L2-E-L3-ext}
0\to L_2\to E\to L_3\to 0
\end{equation}
where $\deg(L_2)=2$, $\deg(L_3)=3$.
We mimic the proof of Proposition \ref{FE-irr-g2-prop}. Set $L:=L_3L_2^{-1}$, and let $e\in H^1(L^{-1})\simeq H^0(KL)^*$ denote
the class of extension \eqref{L2-E-L3-ext}.
Let $S\sub F_E$ denote the closed subset of $(\xi,s\in \P H^0(\xi E))$ such that $H^0(\xi L_2)\neq 0$. We will check that $\dim S\le 2$ and that $F_E\setminus S$
is irreducible of dimenison $3$. This will imply that $F_E$ is irreducible (since we know that every irreducible component of $F_E$ has dimension $\ge g=3$).

For $\xi\in F_E\setminus S$ we have $H^0(\xi L_2)=H^1(\xi L_2)$, so the natural map $H^0(\xi E)\to H^0(\xi L_3)$ is an isomorphism. Thus, by looking at zeros of
the corresponding global section of $\xi L_3$, we identify $F_E\setminus S$ with an open subset of $C^{[3]}$. This proves that  $F_E\setminus S$
is irreducible of dimenison $3$. 

Next, let us check that $\dim S\le 2$. 
Now we consider separately several strata in $S$. Assume first that $h^0(\xi L_2)=2$. Then $\xi L_2$ is the hyperelliptic system, so $\xi$ is uniquely determined.
Since by Lemma \ref{g3-EDp-lem}, $h^0(\xi E)\le 3$ (which is equivalent to $h^1(\xi E)\le 2$), this stratum has dimension $\le 2$.
Thus, we are reduced to considering $(\xi,s)$ in $S$ such that $\xi L_2\simeq \OO(D)$, where $h^0(D)=1$, so $\xi E$ fits into an extension
$$0\to \OO(D)\to \xi E\to L(D)\to 0.$$

\noindent
{\bf Case 1}. First, we have the open locus where $h^0(L(D))=1$ and the connecting homomorphism $H^0(L(D))\to H^1(\OO(D))$ is nonzero.
On this locus $h^0(\xi E)=1$, so a point of this locus is uniquely determined by $D\in C^{[2]}$. Hence, the dimension is $\le 2$.

\noindent
{\bf Case 2}. Next, consider the locus where $h^0(L(D))=1$ and the connecting homomorphism $H^0(L(D))\to H^1(\OO(D))$ is zero.
Since in this case $h^0(\xi E)=2$, it is enough to prove that there is at most $1$-dimenisonal family of such $D$.
The above connecting homomorphism vanishes if and only if the composition
$$H^0(L(D))\ot H^0(K(-D))\rTo{b_{L,D}} H^0(KL)\rTo{e} k$$
is zero. Since $h^0(D)=h^1(D)=1$, we have the unique $D'\in C^{[2]}$ such that $K(-D)\simeq \OO(D')$. 
The image of $b_{L,D}$ is exactly the $1$-dimensional subspace $H^0(KL(-D'))\sub H^0(KL)$.
Thus, it suffices to prove that there exists $D'\in C^{[2]}$ such that $H^0(KL(-D'))$ is $1$-dimensional and is not contained in $\ker(e)$.
The linear system $|\ker(e)|$ gives a rational map $f$ from $C$ to $\P^1$. Let $p_1,p_2\in C$ be a pair of points where $f$ is defined
such that $f(p_1)\neq f(p_2)$. Then $H^0(KL(-p_1-p_2))\cap \ker(e)=0$, so $D'=p_1+p_2$ works.

\noindent
{\bf Case 3}. Now consider the stratum where $h^0(L(D))=2$ and the connecting homomorphism $H^0(L(D))\to H^1(\OO(D))$ is nonzero.
Then $L(D)\simeq K(-p)$ for some $p\in C$, so there is a $1$-dimensional family of such $D$, and $h^0(\xi E)=2$, so the dimenison of
this stratum is $\le 2$.

\noindent
{\bf Case 4}. Finally, consider the stratum where $h^0(L(D))=2$ and the connecting homomorphism $H^0(L(D))\to H^1(\OO(D))$ is zero.
Then $h^0(\xi E)=3$, so we need to prove that there is only finitely many $D$ like this.
The image of $b_{L,D}$ is the $2$-dimensional subspace $H^0(KL(-D'))\sub H^0(KL)$, where $K(-D')\simeq \OO(D)$, and our vanishing condition means
that this subspace coincides with $\ker(e)$.
Assume first that $h^0(L)=0$. Then
the linear system $|KL|$ defines a regular morphism $f:C\to \P^2$ of degree $5$, and $e$ corresponds to a point $[e]\in \P^2$.
If $H^0(KL(-D'))=\ker(e)$ then every point in the support of $D'$ lies in the fiber $f^{-1}([e])$, so there is finitely many possibilities for $D'$.
There remains the case $L=\OO(q)$ for some $q\in C$. Then we have a $2$-dimensional subspace $H^0(K(q-D'))\sub H^0(K(q))=H^0(K)$ that coincides with
$\ker(e)$. This means that $D'=q+p$ and $H^0(K(-p))$ is the fixed $2$-dimensional subspace of $H^0(K)$, which determines $p$ and hence $D'$ up to two
choices (uniquely if $C$ is non-hyperelliptic).
\end{proof}

\subsection{Proof of Conjecture B for non-hyperelliptic curves of genus $3$}

Let $C$ be a curve of genus $3$.
Let us study singularities of $F_E$ for the family of bundles $E_{D,p}$ defined in Lemma \ref{g3-EDp-lem}. 
We will use the following more explicit construction of these bundles.
Let $E=E_{D,p}$ and set $F=E^\vee K$. Then $F$ is a nontrivial extension 
\begin{equation}\label{FDp-extension}
0\to \OO(p)\to F\to \OO(D)\to 0,
\end{equation}
whose push-out under $\OO(p)\to \OO(p+D)$ splits, or equivalently the restriction to $\OO\sub \OO(D)$ splits. We fix such a splitting and
realize $F$ as the kernel
\begin{equation}\label{g3-E-deg3-def-seq}
0\to F\to \OO(p+D)\oplus \OO(D)\rTo{(1,\varphi)} \OO(p+D)|_D\to 0,
\end{equation}
where $\varphi\in \Hom(\OO(D),\OO(p+D)|_D)$ is any element not proportional to the natural projection $\OO(D)\to \OO(p+D)\to \OO(p+D)|_D$.

\begin{lemma}\label{two-maps-varphi-lem}
In the above situation, assume that either $h^0(D+p)=1$ or $C$ is non-hyperelliptic. Then
$H^0(K|_D)$ is spanned by the images of the maps
\begin{align*}
&r_1:H^0(K(-p))\to H^0(K(-p)|_D)\rTo{1} H^0(K|_D) \ \text{ and  }\\
&r_\varphi:H^0(K(-p))\to H^0(K(-p)|_D)\rTo{\varphi} H^0(K|_D).
\end{align*}
In particular, $r_1$ and $r_\varphi$ are not proportional.
\end{lemma}

\begin{proof}
Since $H^0(\OO(p)|_D)$ is spanned by $1$ and $\varphi$, the natural map of sheaves
$$K(-p)|_D\oplus K(-p)|_D\to K|_D$$
is surjective. If $h^0(D+p)=1$ then $H^0(K(-D-p))=0$, so the restriction map $H^0(K(-p))\to H^0(K(-p)|_D)$ is
an isomorphism. If $C$ is non-hyperelliptic then $K(-p)$ is globally generated. In either case we see that the sections $H^0(K(-p))$ generate 
$K(-p)|_D$, so the first assertion follows. The second follows from this and from the fact that $r_1$ has rank $1$ (as it vanishes on $H^0(K(-D))\sub H^0(K(-p))$.
\end{proof}

\begin{lemma}\label{g3-genericE-lemma}
Assume that the effective divisor $D$ of degree $2$ and a point $p$ on a curve $C$ of genus $3$ are such that $h^0(\OO(D+p))=1$. 
Set $E=E_{D,p}$. Then
$F_E$ has rational singularities over $\OO\in J$.
\end{lemma}

\begin{proof} Set $F=E^\vee K$. 
The scheme $F_E$ locally near $\OO$ has form $X=X(\phi)$ (see Sec.\ \ref{deg-loci-sec}) for an $2\times 3$-matrix 
$\phi$ of functions on $J$,
vanishing at $0$, whose linear terms are given by the $2\times 3$-matrix describing the multiplication map
$$\mu: H^0(F)\ot H^0(E)\to H^0(K).$$
Since $h^0(\OO(D))=1$, the exact sequences for $E$ and $F$ induce exact sequences
$$0\to H^0(\OO(p))\to H^0(F)\to H^0(\OO(D))\to 0,$$
$$0\to H^0(K(-D))\to H^0(E)\to H^0(K(-p))\to 0.$$
Note that the splitting $\OO\to F$ induces a splitting of the first sequence, and we denote by $g\in H^0(F)$ the corresponding element.

If we let the first (resp., second) row of the matrix of $\mu$ to correspond to $1\in H^0(\OO(p))\sub H^0(F)$ (resp., $g\in H^0(F)$), then its entries correspond to
the natural projection $\phi_1:H^0(E)\to H^0(K(-p))\sub H^0(K)$ (resp., to the map $\phi_2:H^0(E)\to H^0(K)$ induced by the splitting $E\to K$ of the extension defining $E$).

\noindent
{\bf Case 1}. $p\not\in D$.

Then the assumption that $h^0(\OO(D+p))=1$, which is equivalent to $h^0(K(-D-p))=0$, implies that 
we have a decomposition
$$H^0(K)=H^0(K(-p))\oplus H^0(K(-D)).$$
Thus, if we let $x,y$ be a basis of $H^0(K(-p))$, $z$ a basis of $H^0(K(-D))$, and let the first two columns of $\mu$ correspond to liftings of $x$ and $y$ to $H^0(E)$,
while the last column correspond to $z$, then
the $2\times 3$-matrix describing $\mu$ will be of the form
$$\left(\begin{matrix} x & y & 0 \\ a & b & z \end{matrix}\right),$$
Note that the $2\times 2$ submatrix corresponding to the first two columns corresponds to the restriction of
$\mu$ to $H^0(F)\ot \lan \wt{x},\wt{y}\ran$, where $\wt{x},\wt{y}\in H^0(E)$ are liftings of $x,y\in H^0(K(-p))$.

Now we claim that the vector $(a,b) \mod (z)$ in $H^0(K)/(z)\oplus H^0(K)/(z)$ is not proportional to $(x,y)$.
Indeed, since $z$ is a generator of $H^0(K(-D))\sub H^0(K)$, we can identify
$a \mod z$ and $b \mod z$ with the images of $g\ot \wt{x}$ and $g\ot \wt{y}$ under the composition
\begin{equation}\label{mu-g-restriction}
\lan g \ran \ot H^0(E)\rTo{\mu} H^0(K)\to H^0(K|_D).
\end{equation}

The restriction of $\mu$ to $\lan g\ran \ot H^0(E)$ corresponds to the map $H^0(E)\to H^0(K)$ associated with the splitting of the extension defining $E$ 
under $K(-D)\to K$. The corresponding exact sequence
$$0\to E\to K\oplus K(-p)\rTo{(1,\varphi)} K|_D\to 0,$$
allows to realize $H^0(E)$ as the subspace of $H^0(K)\oplus H^0(K(-p))$ consisting of $(s,t)$ such that
$s|_D=\varphi\cdot t|_D$. 
Thus, the composition \eqref{mu-g-restriction} is given by
$$\lan g\ran \ot (s,t)\mapsto s|_D=\varphi\cdot t|_D,$$
so it factors as the composition $H^0(E)\to H^0(K(-p))\rTo{\varphi} H^0(K|_D)$.

On the other hand, the restriction of $\mu$ to $\lan 1\ran \ot H^0(E)$ is the natural projection $H^0(E)\to H^0(K(-p)): (s,t)\mapsto t$.
Thus, our claim follows immediately from Lemma \ref{two-maps-varphi-lem}.

Using elementary column operations, we can assume that $a$ and $b$ are in $\lan x,y\ran$.
Rescaling the variables we get a family of matrices $\phi_t$ over $\A^1$, with $\phi_t$ equivalent to $\phi$ for $t\neq 0$,
and $\phi_0=\mu$, the matrix of linear terms. 

We claim that $\mu$ satisfies the assumptions of Proposition \ref{deg-blow-up-prop}(iii), i.e., the locus $Z(\mu)$ where $\mu$ is degenerate has codimension $2$ in $\A^3$,
while the vanishing locus of $\mu$ is $0$-dimensional. The latter is clear, while the former follows from the fact that the ideal of $Z(\mu)$ is generated
by $xz$, $yz$ and $xb-ya$, which is a nonzero quadratic form in $x,y$. Hence, By Lemma \ref{flatness-lem}, we have a flat degeneration of $F_E=X(\phi)$ into
$X(\mu)$. By Elkik's Theorem \cite[Thm.\ IV]{Elkik}, it is enough to prove that $X(\mu)$ has rational singularities over the origin. Furthermore, by Proposition   \ref{deg-blow-up-prop}(iii),
$X(\mu)$ is the blow up of $\A^3$ at $Z(\mu)$.
Now the assertion follows from Lemma \ref{rat-sing-quadratic-lem}(ii) below.

\noindent
{\bf Case 2}. $p\in D$.

Note that in this case $\varphi|_p\neq 0$ in $\OO(p)|_p$ and $H^0(K(-D))\sub H^0(K(-p))$. 
We can choose a basis $(x,y)$ on $H^0(K(-p))$ so that $H^0(K(-D))$ is spanned by $x$. Then
arguing as before, we get that $\mu$ takes form
$$\left(\begin{matrix} x & y & 0 \\ a & b & x \end{matrix}\right).$$
As before we see that $a|_D=\varphi\cdot x|_D$, $b|_D=\varphi\cdot y|_D$. 

{\bf Subcase 2a.} $D=p+q$, where $q\neq p$.
Since $H^0(K(-2p-q))=0$ we see that $x|_p\neq 0$. Hence $a|_p\neq 0$, so $a\not\in H^0(K(-p))=\lan x,y\ran$.
Then we can take $a=z$ to be the third basis element in $H^0(K)$.
Let us write $b=\a x+\b y+\ga z$. Then we can perform the following elementary operations
$$\left(\begin{matrix} x & y & 0 \\ z & \a x+\b y+\ga z & x \end{matrix}\right)\mapsto
\left(\begin{matrix} x & y & 0 \\ z-\b x & \a x+\ga z & x \end{matrix}\right)\mapsto
\left(\begin{matrix} x & y-\ga x & 0 \\ z-\b x & (\a+\b\ga) x & x \end{matrix}\right)\mapsto
\left(\begin{matrix} x & y-\ga x & 0 \\ z-\b x & 0 & x \end{matrix}\right),
$$
so changing the variables, we can assume 
$$\mu=\left(\begin{matrix} x & y & 0 \\ z & 0 & x \end{matrix}\right).$$
This matrix satisfies the assumptions of Proposition \ref{deg-blow-up-prop}(iii), so as in Case 1, the assertion reduces to Lemma \ref{rat-sing-quadratic-lem}(ii).

{\bf Subcase 2b.} $D=2p$. Since $y\in H^0(K(-p))\setminus H^0(K(-2p))$, we have $y|_p\neq 0$. Hence $b|_p\neq 0$, so
we can take $b=z$. On the other hand, $x|_p=0$ in $K(-p)|_p$, hence $a|_p=0$, so $a\in \lan x,y\ran$. 
Note also, the since $H^0(K(-3p))=0$ by assumption, the restriction map $H^0(K(-p))\to K(-p)|_{2p}$ is an isomorphism, so $x|_{2p}\neq 0$.
Hence, $a|_{2p}\neq 0$, so $a\not\in \lan x\ran$.
Adding a multiple of the third column to the first,
we can assume that
$$\mu=\left(\begin{matrix} x & y & 0 \\ \b y & z & x \end{matrix}\right).$$
for some $\b\neq 0$. Rescaling the second row, the last column and $z$, we get
$$\mu=\left(\begin{matrix} x & y & 0 \\ y & z & x \end{matrix}\right),$$
and the assertion follows from Lemma \ref{rat-sing-quadratic-lem}(iii).
\end{proof}

\begin{lemma}\label{rat-sing-quadratic-lem}
\noindent
(i) Consider the hypersurface $H\sub \A^n$ (where $n\ge 3$) with the equation
$$x_1x_2=x_1M_1+x_2^dM_2,$$
where $d\ge 0$ and $M_1,M_2$ are monomials in $(x_i)_{i\ge 3}$.
Then $H$ has rational singularities.

\noindent
(ii) The blow up of $\A^3$ at a CM singularity of codimension $2$ given by the ideal $(xz,yz,q(x,y))$ where $q(x,y)$ is a nonzero quadratic form, has rational singularities over the origin.

\noindent
(iii) The variety $X(\mu)=B_{Z(\mu)}\A^3$ for 
$$\mu=\left(\begin{matrix} x & y & 0 \\ y & z & x \end{matrix}\right)$$
has rational singularities.
\end{lemma}

\begin{proof}
(i) Let us rewrite the equation as $x_1(x_2-M_1)=x_2^dM_2$. Changing $x_2$ to $x_2+M_1$, we get $x_1x_2=(x_2+M_1)^dM_2$.
The latter hypersurface fits into an isotrivial family 
$$x_1x_2=(tx_2+M_1)^dM_2.$$
Since the limiting hypersurface $x_1x_2=M_1^dM_2$ is normal toric, it has rational singularities.
Hence, by Elkik's Theorem (see \cite[Thm.\ IV]{Elkik}), so does the original hypersurface.

\noindent
(ii) Changing variables we reduce to considering the ideals,
$$I_1=(xz,yz,xy), \ \ I_2=(xz,yz,x^2),$$
which are degeneration ideals of the matrices
$$\phi_1=\left(\begin{matrix} x & y & 0 \\ 0 & y & z \end{matrix}\right) \ \text{ and }
\phi_2=\left(\begin{matrix} x & y & 0 \\ 0 & x & z \end{matrix}\right).$$
By Proposition \ref{deg-blow-up-prop}(iii), we have identitications of $X(\phi_i)$ with the corresponding blow ups.

By definition $X(\phi_1)$ is the closed subscheme of $\A^3\times \P^2$ given by the equations
\begin{align*}
&x\a+y\b=0\\
&y\b+z\ga=0,
\end{align*}
where $(\a:\b:\ga)$ are homogeneous coordinates on $\P^2$. Over each standard affine open in $\P^2$ we get a normal toric hypersurface,
so $X(\phi_1)$ has rational singularities. For $X(\phi_2)$ the equations are
\begin{align*}
&x\a+y\b=0\\
&x\b+z\ga=0,
\end{align*}
which still gives normal toric hypersurfaces on the standard affine opens, the most interesting one being 
$$y\b^2=z\ga$$
obtained for $\a=1$.


\noindent
(iii) The equations of $X(\mu)$ are
\begin{align*}
&x\a+y\b=0\\
&y\a+z\b+x\ga=0.
\end{align*}
Clearly, the open corresponding to $\b=1$ is smooth. For $\a=1$, we get the hypersurface
$$y+z\b=y\b\ga,$$
which has rational singularities (e.g., by part (i)).
Finally, for $\ga=1$, we get the hypersurface
$$y\b=(y\a+z\b)\a$$
which has rational singularities by part (i).
\end{proof}

\begin{remark}
Another way to prove rationality of singularities in Lemma \ref{rat-sing-quadratic-lem}(ii), is to observe that the blow ups of $I_i$ are toric, so it is enough to
check that the Rees algebras associated with the ideals $I_1$ and $I_2$ are integrally closed. 
This is equivalent to checking that all powers of these ideals are integrally closed. Since we have $3$ independent variables,
by Theorem of Reid-Roberts-Vitulli (see \cite[Thm.\ 1.4.10]{SH}) it is enough to check that $I_i$ and $I_i^2$ are integrally closed for $i=1,2$.
This can be easily done using the combinatorial recipe \cite[Prop.\ 1.4.6]{SH} for finding generators of the integral closures of monomial ideals.
\end{remark} 


Below we will need the following general statement.

\begin{lemma}\label{union-deg-locus-lem}
Let 
$$0\to F\to G\to H\to 0$$
be an exact sequence of vector bundles on $S\times C$, where $S$ is some base scheme and $C$ is a curve.
Let $Z_F$ (resp., $Z_G$, $Z_H$) be the locus in $S$ where $H^0(F_s)\neq 0$ (resp., $H^0(G_s)\neq 0$, resp., $H^0(H_s)\neq 0$) with its natural subscheme structure.
Then we have an inclusion of schemes
$$Z_G\sub Z_F\cup Z_H$$
where the scheme structure of $Z_F\cup Z_H$ is defined using the product of the corresponding ideal sheaves,
\end{lemma}

\begin{lemma}\label{g3-EDp-lemma}
Assume that the effective divisor $D$ of degree $2$ and a point $p$ on a 
non-hyperelliptic curve $C$ of genus $3$ are such that $H^1(\OO(D+p))\neq 0$.
Set $E=E_{D,p}$. Then $F_E$ has rational singularities over $\OO\in J$. 
\end{lemma}

\begin{proof}
We keep the notation above. 
Note that the condition $H^1(\OO(D+p))\neq 0$ is equivalent to $H^0(K(-D-p))\neq 0$, i.e., $H^0(K(-D-p))=H^0(K(-D))$ (since the latter space is $1$-dimensional),
which implies the inclusion 
$$H^0(K(-D))\sub H^0(K(-p)).$$

\noindent
{\bf Step 1}. Calculation of $\mu:H^0(F)\ot H^0(E)\to H^0(K)$.

Recall that the extension $E$ is determined by an element $\varphi\in H^0(\OO(p)|_D)$, not proportional to $1\in H^0(\OO(p)|_D)$, and we identify $H^0(E)$ 
with the set of pairs $(s,t)\in H^0(K)\oplus H^0(K(-p))$ such that $s|_D=\varphi\cdot t|_D$ in $H^0(K|_D)$.

Since $C$ is non-hyperelliptic, 
by Lemma \ref{two-maps-varphi-lem}, $H^0(K|_D)$ is spanned by the images of the maps
\begin{equation}\label{two-images-eq}
H^0(K(-p))\to H^0(K(-p)|_D)\rTo{1} H^0(K|_D) \ \text{ and  } H^0(K(-p))\to H^0(K(-p)|_D)\rTo{\varphi} H^0(K|_D)
\end{equation}
But both maps vanish on $H^0(K(-p-D))=H^0(K(-D))$, so have at most $1$-dimensional images. It follows that the kernel of both maps is $H^0(K(-D))$, both images are $1$-dimensional, and they span $H^0(K|_D)$.

Now we can choose bases of $H^0(E)$ and $H^0(K)$ in a special way. First, let $s_1$ by a generator of $H^0(K(-D))$, which we also view as an element $(s_1,0)$ of $H^0(E)$.
Next, let $s_2$ be an element of $H^0(K(-p))\setminus H^0(K(-D))$. Let us pick an element $s_3\in H^0(K)$ such that
$$s_3|_D=\varphi s_2|_D.$$
Then as we have shown above, $\varphi s_2|_D\neq 0$, so $s_3\not\in H^0(K(-p))$ (otherwise, we would get nonzero intersection of the images of the maps \eqref{two-images-eq}).
Hence, $s_1,s_2,s_3$ is a basis of $H^0(K)$. Also,
$$(0,s_1), \ (s_3,s_2), \ (s_1,0)$$
is a basis of $H^0(E)$.
Thus, if we use the basis $1,g$ of $H^0(F)$, we get the following matrix for $\mu:H^0(F)\ot H^0(E)\to H^0(K)$:
$$\left(\begin{matrix} s_1 & s_2 & 0 \\ 0 & s_3 & s_1 \end{matrix}\right).$$
Thus, the quadratic parts of the generators of the ideal of $Z$ are $s_1^2$, $s_1s_2$ and $s_1s_3$.


\noindent
{\bf Step 2}. Let us consider the embedding $\iota_p:C\hra J: q\mapsto \OO(q-p)$, and let $I_{\iota_p(C)}\sub \OO_J$ be the corresponding ideal. 
By Lemma \ref{union-deg-locus-lem} applied to the exact sequence \eqref{FDp-extension} we get
$$\th_D\cdot I_{\iota_p(C)}\sub I_Z,$$
where $\th_D$ is the local equation of the theta divisor $\Th_D=\{D'-D \ | D'\in C^{[2]}\}$.


\noindent
{\bf Step 3}. It is known that $s_1$ (the generator of $H^0(K(-D))$)
 is proportional to the first derivative of $\th_D$ at zero, and that $(s_1,s_2)^\perp$ is the tangent space to $\iota_p(C)\sub J$ at zero.
Thus, we can choose a formal system of coordinates $(x,y,z)$ at $0\in J$ with $x=\th_D$ such that the ideal of $C$ has generators $(x+h,y)$ with $h\in \fm^2$.
We claim that possibly changing $z$, we can assume the ideal $I_Z$ to be 
\begin{equation}\label{g3-IZ-eq}
(x(x+h),xy,(x+h)z+yf),
\end{equation}
where $f\in \fm^2$. 

Indeed, by Step 2 we know that $x(x+h)$ and $xy$ are in $I_Z$, whereas
from Step 1 we know that $I_Z$ is generated by $f_1=x^2+\ldots$, $f_2=xy+\ldots$ and $f_3=xz+\ldots$. 

Thus, 
$$x^2=a_1f_1+a_2f_2+a_3f_3.$$
Looking at quadratic terms we see that $a_1$ is invertible. Hence, we can replace $f_1$ by $x(x+h)$ in the set of generators of $I_Z$. Similarly,
we can replace $f_2$ by $xy$.

Also, from the inclusion $\OO(-p)\sub F$ we see that $I_Z$ is contained in the ideal $(x+h,y)$ of $\iota_p(C)$.
Hence, we can write $f_3=(x+h)g_1+yg_2$. Since the quadratic part of $f_3$ is $xy$, the linear part of $g_1$ is $z$ and $g_2\in \fm^2$.
Thus, $f_3=(x+h)(z+g)+yf$, where $f,g\in \fm^2$. Since $x(x+h)\in I_Z$, we can add a multiple of $x$ to $g$. Also, we can add multiples of $y$ to $g$ by changing
$f$. Hence, we can assume that $g$ depends only on $z$, and then change $z$ to $z+g(z)$.

\noindent
{\bf Step 4}. Next, we would like to compute the scheme-theoretic intersection $Z\cap \Th_D$ near $0$. From the form of the ideal we got in Step 4, 
we know that this a Cartier divisor in $\Th_D$.
We begin by compute this locus set-theoretically. 

Note that $\Th_D$ consists of $\xi$ of the form $\OO(D'-D)$, where $D'$ is an effective divisor of degree $2$.
Thus, we have an identification $\Th_D=C^{[2]}$, and $Z\cap \Th_D$ is a divisor in $C^{[2]}$.
We have a family of bundles over $C^{[2]}$
\begin{equation}\label{ED'-ex-seq}
0\to K(-D')\to E_{D'}\to K(D-p-D')\to 0
\end{equation}
with the fixed nonzero extension class $e\in H^1(\OO(p-D))$ which vanishes in $H^1(\OO(p))$,
and we are interested in the locus where $H^1(E_{D'})\neq 0$. 

The long exact sequence of cohomology shows that $H^1(E_{D'})\neq 0$ if either
(a) $H^1(K(D-p-D'))\neq 0$ or (b) $H^1(K(D-p-D'))=0$ and the cup product with $e$ gives a map $H^0(K(D-p-D'))\to H^1(K(-D'))$ of non-maximal rank.

Condition (a) is equivalent to $K(D-p-D')\simeq K(-q)$ for some $q\in C$, i.e., $D'+p\sim D+q$.
Let $D+q_1+q_2$ be the unique canonical divisor containing $D$. Assume $D'\neq D$, so $q\neq p$.
If $q\neq q_1,q_2$ then $h^0(D+q)=1$, so we should have $D'+p=D+q$. This is possible only if $D=p+q_0$ and $D'=q+q_0$ for some $q_0\in C$.
If $q=q_1$ then $D'$ is uniquely determined by $\OO(D')\simeq K(-q_2-p)$.
We can ignore the cases giving a finite number of possibilities for $D'$, since $Z\cap\Th_D$ is a divisor in $\Th_D$. Thus, condition (a) gives a component only in the case
$D=p+q_0$ consisting of $D'$ of the form $q+q_0$.

On the other hand, since $e$ corresponds to the linear functional on $H^0(K(D-p))$ vanishing on $H^0(K(-p))$, condition (b) states that $H^0(K(D-p-D'))$ is $1$-dimensional
and we have an inclusion $H^0(K(D-p-D'))\sub H^0(K(-p))$ of subspaces in $H^0(K(D-p))$, or $\OO(D')$ is the hyperelliptic system. The latter case gives an isolated point of $Z$, so
we can ignore it. 
If the supports of $D'$ and $D$ are disjoint then we have
$$H^0(K(D-p-D'))\cap H^0(K(-p))=H^0(K(-p-D')),$$
so this space should be nonzero. This implies that $K(-D')\sim p+q$ for some $q\in C$.
Conversely, if this is the case and $h^0(\OO(D+q))=1$ then condition (b) is satisfied (with no extra assumption on the supports).

Next, assume that $D=p_1+p_2$, and $D'=p_1+q$. Then condition (b) states that $H^1(K(p_2-p-q))=0$ and $H^0(K(p_2-p-q))\sub H^0(K(-p))$.
In other words, this holds if and only if $p_2\neq p$, $p_2\neq q$.

It follows that set-theoretically $Z\cap \Th_D$ is the union of three copies of the curve embedded into $\Th_D=C^{[2]}$:
$$C_1=p_1+C, \ \ C_2=p_2+C, \ \ C':=\{D'\sim K(-p-q) \ | \ q\in C\}$$
(where the first two subsets are the same if $p_1=p_2$). Note that $D$ is a point in $C'$ since $H^0(K(-D-p))\neq 0$.
The corresponding subsets in $J$ are
$$\iota_{p_2}(C)=C-p_2, \ \ \iota_{p_1}(C)=C-p_2, \ \ K-p-D-C=-\iota_{p_0}(C),$$
where $D+p+p_0\sim K$.


\noindent
{\bf Step 5}. We claim that the equation of $Z\cap \Th_D$ in $\Th_D=C^{[2]}$
is given by the determinant of the morphism of vector bundles on $C^{[2]}$,
$$H^0(K(D-p))\rTo{(1,e)} H^0(K(D-p)|_{D'})\oplus k,$$
where the functional $e$ corresponds to our class in $H^1(\OO(p-D))$ (so it vanishes on $H^0(K(-p))\sub H^0(K(D-p))$).
Equivalently, it is given by the determinant of the composed morphism
$$H^0(K(-p))\rTo{\a} K(-p)|_{D'}\rTo{\b_1} K(p_1-p)|_{D'}\rTo{\b_2} K(D-p)|_{D'}.$$
Thus, the equation of $Z\cap \Th_D$ in $\Th_D$ is given by $\det(\a)\det(\b_1)\det(\b_2)$.

To prove this, note that the exact sequence \eqref{ED'-ex-seq} over $C\times C^{[2]}$ gives rise to an exact triangle on $C^{[2]}$
$$R\pi_*(K(-D'))\to R\pi_*(E_{D'})\to R\pi_*(K(D-p-D'))\to \ldots$$
where $\pi:C\times C^{[2]}\to C^{[2]}$ is the projection.
Next, we observe that $L_0=R^0\pi_*(K(-D'))$ and $L_1=R^1\pi_*(K(-D'))$ are line bundles on $C^{[2]}$. Since the map $H^0(K(-D'))\to H^0(E_{D'})$ is always an
embedding, the dual map is a surjection, so considering the dual map to composition $L_0\to R\pi_*(K(-D'))\to R\pi_*(E_{D'})$, we see that locally we have a splitting
$$R\pi_*(E_{D'})=L_0\oplus P^\bullet,$$
where $P^\bullet$ fits into an exact triangle
$$P^\bullet\to R\pi_*(K(D-p-D'))\rTo{\de} L_2\to\ldots$$
where $\de$ is enduced by our class in $H^1(\OO(p-D))$.
Note that the equation $Z\cap \Th_D$ is obtained as the determinant of the differential $P^0\to P^1$ in the complex $P^\bullet$ (we can take any representative).

Now to prove our claim it remains to show that the map $\de$ is represented by the chain map
$$[\pi_*(K(D-p))\to \pi_*(K(D-p)|_{D'})]\rTo{e} \OO_{C^{[2]}}.$$
Note the following square in the derived category is commutative
\begin{diagram}
K(D-p-D')&\rTo{}&K(D-p)\\
\dTo{e} &&\dTo{e}\\
K(-D')[1]&\rTo{}&K[1]
\end{diagram}
Passing to $R\pi_*$ we see that $\de$ corresponds to the composition
$$R\pi_*(K(D-p-D'))\to R\pi_*(K(D-p))\rTo{e} R\pi_*K[1]\to H^1(K)\ot \OO_{C^{[2]}}.$$
Since the first arrow in this composition is represented by the chain map 
$$[\pi_*(K(D-p))\to \pi_*(K(D-p)|_{D'})]\to \pi_*(K(D-p)),$$
our claim follows.

\noindent
{\bf Step 6}. We claim that $\det(\a)$ generates the ideal of the component $C'\sub C^{[2]}$ consisting of $D'$ such that $H^0(K(-p-D'))\neq 0$,
while $\det(\b_i)$ generates the ideal of $C_i$, for $i=1,2$. 

For $C_i$, this follows from a simple local calculation. Let $t$ be an \'etale local coordinate on $C$ at $p_i$ (so that $t=0$ corresponds to $p_i$). We can model \'etale 
locally $C^{[2]}$ as the space of polynomials $t^2+at+b$, so that $C_i$ is given by $b=0$. Now our assertion follows from the fact that the determinant of the multiplication by $t$
on $k[t]/(t^2+at+b)$ is equal to $b$ (in some basis).

For $C'$, we first observe that it is enough to make a calculation at a generic point, so we can consider $D'$ disjoint from $p$. Then we can replace $\a$ with the map
$H^0(K)\to K|_{D'}\oplus K|_p$ and then with the map
$$H^0(K(-D'))\to K(-D')|_p.$$
We have an involution $\tau$ on $C^{[2]}$ defined by $D'+\tau(D')\sim K$, and locally we can choose an isomorphism $\OO(D'+\tau(D'))\simeq K$.
Thus, the above map can be replaced by the map $H^0(\OO)\to H^0(\OO(\tau(D')))\to \OO(\tau(D'))|_p$, or equivalently by the map
$$\OO|_p\to \OO(\tau(D'))|_p.$$
As before, the corresponding equation gives precisely the (reduced) locus of $\tau(D')$ such that $\tau(D')=p+q$ for some $q\in C$. 


\noindent
{\bf Step 7}. 
Recall that $I_Z$ is given by \eqref{g3-IZ-eq}, i.e., it is the degeneracy locus of the matrix
$$\phi=\left(\begin{matrix} x+h & y & 0 \\ -f & z & x \end{matrix}\right).$$ 
Using the fact that $xy\in I_Z$, we can assume $h=h(x,z)$ and $f=f(y,z)$.
Furthermore, making a change of variables of the form $x\mapsto x(1+c_1z+c_2z^2+\ldots)$, we can assume that $h=h(z)\in \fm^2$.

The scheme $Z\cap \Th_D$ is given in $\Th_D$ by the equation $zh(z)+yf(y,z)$. By Steps 5 and 6, the cubic part of the latter equation is nonzero.
Hence, the assertion follows from Lemma \ref{rat-sing-2-3-deg-lem} below.
\end{proof}

\begin{lemma}\label{rat-sing-2-3-deg-lem}
Consider the scheme $X(\phi)$ for 
$$\phi=\left(\begin{matrix} x+h(z) & y & 0 \\ f(y,z) & z & x \end{matrix}\right),$$ 
where $h,f\in\fm^2$ are such that either $h(z)=cz^2+\ldots$ with $c\neq 0$, or the quadratic part of $f$ is nonzero.
Then $X(\phi)$ is isomorphic to $B_{Z(\phi)}\A^3$ and has rational singularities. 
\end{lemma}

\begin{proof}
First of all, it is easy to check that this $\phi$ satisfies the assumptions of Proposition \ref{deg-blow-up-prop}(iii), hence we get $X(\phi)\simeq B_{Z(\phi)}\A^3$.
Now we consider several cases.

\noindent
{\bf Case 1}. $h\in \fm^3$ and $f=f_2(y,z)+\ldots$, where $f_2(y,z)$ is a nonzero quadratic form. We can include $\phi$ into the $1$-parametric family
\begin{equation}\label{phi-t-case1-eq}
\phi_t=
\left(\begin{matrix} x+t^{-2}h(tz) & y & 0 \\ t^{-2}f(ty,tz) & z & x \end{matrix}\right).
\end{equation}
Note that for $t\neq 0$, $\phi_t$ is equivalent to $\phi(t^2 x,ty,tz)$.
Thus, it suffices to check the assertion for $\psi=\phi_0$, i.e., we reduce to the case $h=0$ and $f=f_2\neq 0$.

{\bf Subcase 1a}. Coefficient of $z^2$ in $f_2$ is nonzero. Then we consider the family
$$\psi_t=
\left(\begin{matrix} x & y & 0 \\ t^{-2}f_2(t^2y,tz) & z & x \end{matrix}\right).$$ 
It is easy to see that for $t\neq 0$, $\psi_t$ is equivalent to $\psi(t^3x,t^2y,tz)$. Hence, it is enough to check the assertion for
$\psi_0$, i.e., we reduce to the case $f_2=c z^2$, where $c\neq 0$. Furthermore, we can replace $\psi_0$ by an equivalent matrix with $f_2=z^2$.

The equations of $X(\psi_0)$ in $\A^3\times\P^2$ are
\begin{align*}
&x\a+y\b=0\\
&z^2\a+z\b+x\ga=0.
\end{align*}
Over $\a=1$, this gives
$$z^2+z\b=y\b\ga,$$
which has rational singularities by Lemma \ref{rat-sing-quadratic-lem}(i).
Over $\b=1$, we get the smooth hypersurface $z^2\a+z+x\ga=0$. Finally, over $\ga=1$,
we get
$$y\b=(z^2\a+z\b)\a,$$
which has rational singularities by Lemma \ref{rat-sing-quadratic-lem}(i).

{\bf Subcase 1b}. $f_2=ayz+by^2$, where $a\neq 0$. Then we consider the family
$$\psi_t=
\left(\begin{matrix} x & y & 0 \\ t^{-3}f_2(t^2y,tz) & z & x \end{matrix}\right).$$ 
Then for $t\neq 0$, $\psi_t$ is equivalent to $\psi(t^4x,t^2y,tz)$. Hence, we reduce to the case of $\psi_0$, i.e., $f_2=yz$.
The equations of $X(\psi_0)$ in $\A^3\times\P^2$ are
\begin{align*}
&x\a+y\b=0\\
&yz\a+z\b+x\ga=0.
\end{align*}
Over $\b=1$, we get the hypersurface $(y\a+1)z=-x\ga$, which fits into the isotrivial family $(y\a+t)z=-x\ga$ with the normal toric limit $yz\a=-x\ga$.
Over $\a=1$ and over $\ga=1$, we get the hypersurfaces
$$yz+z\b=y\b\ga,$$
$$y\b=(y\a+\b)z\a.$$
Both have rational singularities by Lemma \ref{rat-sing-quadratic-lem}(i).

{\bf Subcase 1c}. $f_2=cy^2$, where $c\neq 0$. Then we can replace by an equivalent matrix with $f_2=y^2$.
The affine piece of $X(\psi_0)$ over $\b=1$ is the smooth hypersurface $y^2\a+z+x\ga=0$.
Over $\a=1$ and over $\ga=1$ we get 
$$y^2+z\b=y\b\ga,$$
$$y\b=(y^2\a+z\b)\a,$$
Both have rational singularities by Lemma \ref{rat-sing-quadratic-lem}(i).

\noindent
{\bf Case 2}. $h(z)=c\cdot z^2+\ldots$. Then we can use the family
$$\phi_t=
\left(\begin{matrix} x+t^{-2}h(tz) & y & 0 \\ f(t^3y,tz) & z & x \end{matrix}\right),$$
such that for $t\neq 0$, $\phi_t$ is equivalent to $\phi(t^2x,t^3y,tz)$, to reduce to the case
$h=cz^2$ and $f=0$. Futhermore, we can assume $c=1$.

It is easy to see that the affine piece of $X(\phi_0)$ over $\b=1$ is smooth. Over $\a=1$ and over $\ga=1$
we get hypersurfaces
$$\b z=\ga(z^2+\b\ga),$$
$$\b y=\a(\b z-z^2).$$
Both have rational singularities by Lemma \ref{rat-sing-quadratic-lem}(i).
\end{proof}

\begin{theorem}\label{ConjB-g3-thm}
Conjecture B holds for a non-hyperelliptic curve of genus $3$.
\end{theorem}

\begin{proof}
Let $E$ be a stable rank $2$ bundle of degree $5$ with $h^1(E)\neq 0$. We want to check that $F_E$ has rational singularities over $\OO\in J$.
Assume first that $h^1(E)=1$. 
Since the kernel $A$ of the nonzero map $E\to K$ has degree $\le 2$, we have $h^0(A)\le 1$, and 
our assertion follows from Lemma \ref{rat-sing-easy-case-lem}. Note that the assumption on dimension of $Z_E$ is satisfied by Proposition \ref{Conj-Ci-g3-prop}.

If $h^1(E)>1$ then by Lemma \ref{g3-EDp-lem}, $E$ has form $E_{D,p}$, so the assertion follows from Lemmas \ref{g3-genericE-lemma} and \ref{g3-EDp-lemma}. 
\end{proof}

\end{document}